\documentclass{article}

\usepackage{enumerate}
\usepackage{amsmath}
\usepackage{amsthm}
\usepackage{mathrsfs}
\usepackage{verbatim}
\usepackage{amsfonts}
\usepackage{graphicx}
\usepackage{fullpage}
\usepackage{url}

\usepackage[font={small}]{caption}
\usepackage{titlesec}
\titleformat*{\section}{\large\bfseries}

\newtheorem{theorem}{Theorem}
\newtheorem{proposition}[theorem]{Proposition}
\newtheorem{lemma}[theorem]{Lemma}
\newtheorem{corollary}[theorem]{Corollary}
\newtheorem{definition}[theorem]{Definition}

\newtheorem{assumptions}[theorem]{Assumptions}
\newtheorem{example}[theorem]{Example}
\newtheorem{remark}[theorem]{Remark}

\begin{document}

\title{\Large Measure-transmission metric and stability of structured population models}
\author{G.~Jamr\'oz\\
\small
Institute of Mathematics, Polish Academy of Sciences, \'Sniadeckich 8, 00-956 Warszawa,\\ \small Institute of Applied Mathematics and Mechanics, University of Warsaw, Banacha 2, 02-097 Warszawa\\
\small e-mail: jamroz@mimuw.edu.pl}
\maketitle

\begin{abstract}
In [Gwiazda, Jamr\'oz, Marciniak-Czochra 2012] a framework for studying cell differentiation processes based on measure-valued solutions of transport equations was introduced. Under application of the so-called measure-transmission conditions it enabled to describe processes involving both discrete and continuous transitions. This framework, however, admits solutions which lack continuity with respect to initial data. In this paper, we modify the framework from  [Gwiazda, Jamr\'oz, Marciniak-Czochra 2012] by replacing the flat metric, known also as bounded Lipschitz distance, by a new Wasserstein-type metric. We prove, that the new metric provides stability of solutions with respect to perturbations of initial data while preserving their continuity in time. The stability result is important for numerical applications.

\end{abstract}


\numberwithin{equation}{section}

\section{Introduction}
\label{Sec_intro}
Cell differentiation process is a biological phenomenon, in which immature cells of living organisms give rise to more mature, i.e. more specialized, ones, see e.g. \cite{GILBERT}. In humans, this process takes place primarily during gestation, childhood and adolescence. During these initial stages of human development  a fertilized egg cell, called zygote, divides and differentiates multiple times, giving eventually rise to mature cells of blood, muscles, skin, brain etc. In some tissues, the process of cell differentiation persists during adulthood. 

For instance, neural stem cells or neural progenitors, which reside in the part of brain called hippocampus, can differentiate (Fig. \ref{Fig_neurogenesis}) to become eventually mature neurons, which has implications for human memory, see e.g. \cite{Neurogenesis, Kriegstein}.
\begin{figure}[htbp]
\begin{center}
\includegraphics[width=10cm]{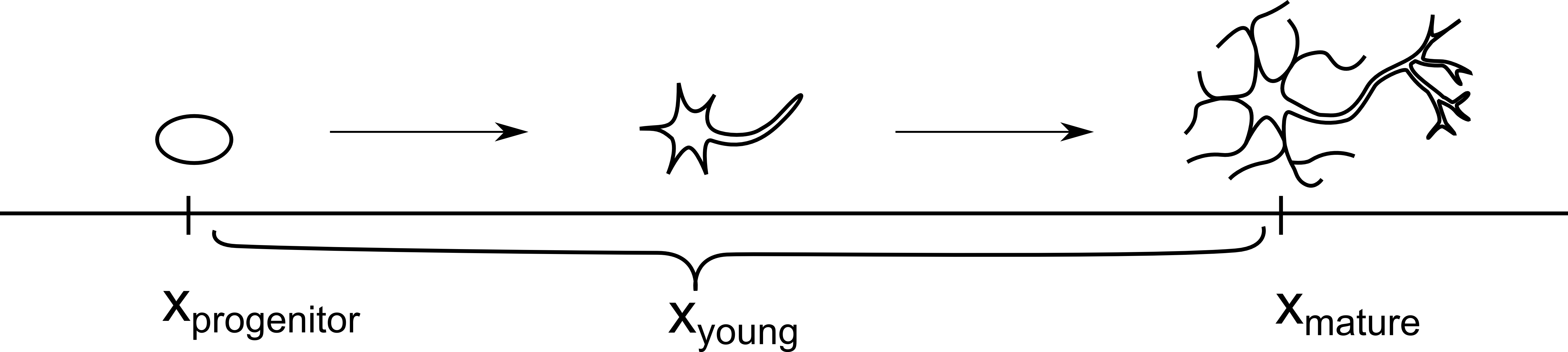}
\caption{Schematic drawing of process of differentiation of neurons in hippocampus. From the discrete state of neural progenitor a cell differentiates to become a young neuron. This continuous phase lasts around four weeks and consists in migration and morphological maturation. Finally, the young neuron reaches the discrete state of maturity.
 }
\label{Fig_neurogenesis}
\end{center}
\end{figure}

Various mathematical models, focusing on different aspects of the process of cell differentiation, and using various mathematical structures, have been proposed in scientific literature. They include modeling differentiation switches via Markov chains or systems of ordinary differential equations (see \cite{GotoKaneko,VillaniSerra,Bodaker}), modeling the inherent stochasticity via branching processes (see e.g. \cite{Corey, Sehl, MacMc}), modeling delays via delay differential equations (see \cite{Adimy0,Arino, Kim} and references therein), modeling spatial dynamics via discrete lattice models or reaction-diffusion equations (see \cite{GranerGlazier, Wang}) and others.

The approach developed in the present paper is called \emph{structured population models}. It consists in tracing populations of cells according to their \emph{maturity level} which is described by a real structure variable $x \in \mathbb{R}$. The order on states $x$ is inherited from $\mathbb{R}$, which means that state $x_2$ is \emph{more differentiated} (i.e. more specialized, more mature) than state $x_1$ iff $x_1<x_2$. This, in turn, means that a cell from state $x_1$ can differentiate into a cell in state $x_2$ yet not vice versa. We distinguish two types of states:
\begin{itemize}
\item \emph{discrete states}, in which cells can stay for a positive period of time (e.g. state of stem cell, state of mature cell),
\item \emph{continuous states}, which cells pass without halting (e.g. the group of states corresponding to maturing neuron).
\end{itemize}
Depending on the topology of the state space we distinguish three basic groups of structured population models of cell differentiation:
\begin{itemize}
\item discrete models, with state space being a finite subset of $\mathbb{R}$ and composed of discrete states only; the dynamics is based on systems of ODEs, see e.g. \cite{Lo, AMC, Nakata, Stiehl},
\item continuous models, with state space being an interval and composed of continuous states only; the evolution of population of cells is then described by a time-dependent density $u(t,x)$ or, more generally, time-depedent positive Radon measure $\mu(t) \in \mathcal{M}(\mathbb{R})$ which evolves according to the transport (balance) equation $\partial_t \mu + \partial_x(g\mu) = p\mu$, see \cite{Adimy,Belair, Colijn, diekmanngetto, Spalding},
\item mixed models, which have both discrete and continuous parts, see \cite{Doumic}. 
\end{itemize}

In \cite{OUR} continuous and mixed models of cell differentiation were embedded into a general framework based on measure-valued solutions of transport equations. We refer to this paper for motivations and further biological background as well as derivation of constituents of the model. Mathematically, framework from \cite{OUR} reads as follows:
\begin{eqnarray}
\partial_t \mu(t) + \partial_x (g_1(v(t))\bold{1}_{x \neq x_i}(x) \mu(t)) &=& p(v(t),x) \mu(t),\label{eq_munonlinold}\\
g_1(v(t))  \frac {D\mu(t)}{D\mathcal{L}^1} (x_i^+) &=& c_i(v(t)) \int_{\{x_i\}} d\mu(t), \label{eq_bcnonlinold} \quad\quad i = 0,\dots,N\\
\mu(0) &=& \mu_0,\label{eq_icnonlinold}
\end{eqnarray}
where $t\in \mathbb R^+$ and  $x\in \mathbb R$.   $x_0 < x_1 < \dots < x_N$ is a finite collection of points in $\mathbb{R}$, which correspond to discrete states. $\bold{1}_{x \neq x_i}$ is equal $1$ if $x \in (x_0,x_1) \cup (x_1,x_2) \cup \dots \cup (x_{N-1},x_N)$ and $0$ otherwise. $\frac {D\mu}{D\mathcal{L}^1}$ denotes the density of measure $\mu$ with respect to the one-dimensional Lebesgue measure and $v(t) := \int_{\{x_N\}} d\mu(t)$ denotes the mass of point $x_N$. The initial datum $\mu_0$ is a Radon measure supported on the interval $[x_0,x_N]$.

Under certain assumptions on coefficients (see \cite[Assumptions 3.2]{OUR}) it was proven that there exists a unique solution 
\begin{equation*}
\mu \in C([0,\infty), (\mathcal{M},\rho_F))
\end{equation*} 
of problem \eqref{eq_munonlinold}-\eqref{eq_icnonlinold}. Here,  $\mathcal{M} = \mathcal{M}(\mathbb{
R})$ is the space of nonnegative Radon measures on $\mathbb{R}$ (see \cite{EVANS_GARIEPY}  for an introduction to measure theory) and $C([0,\infty), (\mathcal{M},\rho_F))$ is the space of continuous functions on $[0,\infty)$ with values in space $\mathcal{M}$  equipped with the flat metric $\rho_F$, which is an adaptation of Wasserstein metric used in the theory of optimal transport, see \cite{VILLANI}. This metric, known also under the name bounded Lipschitz distance, is defined by
\begin{equation}
\label{Eq_defflat}
\rho_F(\mu_1,\mu_2) := \sup_{\psi \in Lip^b(\mathbb{R}), |\psi |\le 1, Lip(\psi) \le 1} \int_{\mathbb{R}}  \psi d(\mu_1 - \mu_2),
\end{equation}
where $Lip^b(\mathbb{R})$ is the set of bounded Lipschitz continuous functions on $\mathbb{R}$ and $Lip(\psi)$ is the Lipschitz constant of $\psi$.

The starting point for the present research is the fact that the space $C([0,\infty), (\mathcal{M},\rho_F))$ is incompatible with the structure of problem \eqref{eq_munonlinold}-\eqref{eq_icnonlinold} in the sense highlighted by the following example. 

\begin{example}[Instability in flat metric]
\label{Ex_stabnonstab}
Take $N=2$ and let $g_1 \equiv 1$ and $c_1 \equiv 0$ in \eqref{eq_munonlinold}-\eqref{eq_icnonlinold}. 
For initial condition $\mu_0 = \delta_{x_1}$ the unique solution of problem \eqref{eq_munonlinold}-\eqref{eq_icnonlinold} in the sense of \cite[Definition 3.3]{OUR} is given by
$$\mu(t) = \delta_{x_1}(dx).$$
Here, $\delta_{x_1}(dx)$ denotes a Dirac mass concentrated in $x_1$.

For a perturbed initial condition $\mu^{\varepsilon}_0 = \delta_{x_1 + \varepsilon}$, on the other hand, we have 
$$\mu^{\varepsilon}(t) = \delta_{x_1 + \varepsilon + t}(dx).$$
Using formula \eqref{Eq_defflat}, we obtain
$\rho_F(\mu(t),\mu^{\epsilon}(t)) = t + \epsilon. $
This means that 
\begin{itemize}
\item $\rho_F(\mu_0,\mu^{\varepsilon}_0) = \varepsilon \to 0$ as $\varepsilon \to 0$,
\item $\rho_F(\mu(t),\mu^{\varepsilon}(t)) = t + \varepsilon \to t$ as $\varepsilon \to 0.$
\end{itemize}
Hence, solutions are neither continuous nor stable with respect to initial data. 
\end{example}
The goal of the present paper is to introduce a new metric, $\rho_{MT}$, which better reflects the structure of system \eqref{eq_munonlinold}-\eqref{eq_icnonlinold} and admits a stability result, which we subsequently prove. 

The paper is organised as follows. In Section \ref{Sec_metrics} we introduce a new metric on Radon measures and discuss its properties. In Section \ref{Sec_Framework}  we present the modified framework of cell differentiation and state the main stability theorem. Section \ref{Sec_Stabilityp0} is devoted to its proof and discussion. Finally, in Appendix we gather additional estimates used in the proofs.

\section{Metrics on the space of measures and measure-transmission metric}
\label{Sec_metrics}

In this section, we study a general class of metrics on Radon measures on $\mathbb{R}$. We discuss and motivate the selection of the one appropriate for system  \eqref{eq_munonlinold}-\eqref{eq_icnonlinold} -- the measure transmission metric $\rho_{MT}$.
\begin{definition}[General class of metrics on $\mathcal{M}(\mathbb{R})$]
Let $\mu_1,\mu_2$ be two finite Radon measures on $\mathbb{R}$. Define
\label{def_rhoGeneral}
\begin{equation}
\label{Eq_def_rhoGeneral}
\rho(\mu_1,\mu_2) :=  \sup_ {\psi \in TFS} \int_{\mathbb{R}}\psi d(\mu_1 - \mu_2),
\end{equation}
where $TFS$ (Test Function Space) is a given subspace of $\mathcal{B}(\mathbb{R})$ (Borel functions on $\mathbb{R}$).
\end{definition}
The most important examples of metrics and their TFSs are summarized in Table \ref{table1}. 
\begin{table}[htbp]
\centering
    \begin{tabular}{ | p{7cm} | p{6.5cm} | p{1.7cm} |}
    \hline
    \textbf{Name of metric}&\textbf{Test Function Space (TFS)} & \textbf{Notation}\\ 
    \hline
Norm (strong) distance& $\{\psi \in \mathcal{B}(\mathbb{R}):  \sup |\psi| \le 1  \}$ & $\|\cdot \|$ 
 \\ \hline

Measure-Transmission metric & Defined below & $\rho_{MT}$
\\  \hline
 
1-Wasserstein distance& $\{\psi \in {\rm Lip}(\mathbb{R}):  {\rm Lip} (\psi) \le 1\}$ & $\rho_{W}$
\\ \hline
Bounded Lipschitz distance or flat metric& $\{\psi \in {\rm Lip}(\mathbb{R}):  {\rm Lip} (\psi) \le 1, \sup|\psi| \le 1  \}$ & $\rho_F$\\
\hline
\end{tabular}
\caption{Metrics on the space of Radon measures and their Test Function Spaces.}
\label{table1}
\end{table}
\begin{proposition}
\label{Prop_ismetric}
Formula \eqref{Eq_def_rhoGeneral} defines a metric provided that TFS satisfies:
\begin{enumerate}[i)]
\item If $\psi \in TFS$ then $-\psi \in TFS$,
\item The set $\{af: f \in TFS,\ 0<a<\infty\}$ contains all smooth compactly supported functions.
\end{enumerate}
\end{proposition}
\begin{proof}
By assumption i) $$\rho(\mu_1,\mu_2) = \rho(\mu_2,\mu_1).$$ Next, if $\mu_1,\mu_2,\mu_3$ are finite Radon measures then $$\int_{\mathbb{R}} \psi d(\mu_1 - \mu_3) = \int_{\mathbb{R}} \psi d(\mu_1 - \mu_2) + \int_{\mathbb{R}} \psi d(\mu_2 - \mu_3).$$ Taking the supremum over $\psi \in TFS,$ we obtain $$\rho(\mu_1,\mu_3) \le \rho(\mu_1,\mu_2) + \rho(\mu_2,\mu_3).$$ Finally, suppose that $\mu_1 \neq \mu_2$. Then $\sigma = \mu_1 - \mu_2$ is a signed measure. From the Hahn-Jordan decomposition theorem (see e.g. \cite[Theorem 4.1.4 and Corollary 4.1.5]{Cohn}) we obtain positive Radon measures $\sigma^+,\sigma^-$ and disjoint Borel sets $N,P$ such that $\sigma^+(N) = 0$, $\sigma^-(P) = 0$ and $\sigma = \sigma^+ - \sigma^-$. Since $\sigma \neq 0$, $\sigma^+(P)>0$ or $\sigma^-(N)>0$. Without loss of generality, assume that $\sigma^+(P)>0$. Then there exists a ball $B(0,R)$ such that $$\sigma^+(P \cap B(0,R))>0.$$ Take $\psi = \bold{1}_{P \cap B(0,R)}$ and $\psi^{\varepsilon} = \psi * \rho^{\varepsilon}$, where $\rho^{\varepsilon}$ is the standard mollifier. We have
\begin{eqnarray*}
\int_{\mathbb{R}} \psi^{\varepsilon} d(\mu_1 - \mu_2) = \int_{\mathbb{R}}{\psi^{\varepsilon}} d\sigma^+ - \int_{\mathbb{R}}{\psi^{\varepsilon}} d\sigma^-.
\end{eqnarray*}
Using the fact that $\psi^{\varepsilon}$ is bounded by $1$ for every $\varepsilon>0$ and $\psi^{\varepsilon} \to \psi$ pointwise, we pass to the limit in all the terms and obtain
\begin{equation*}
\lim_{\varepsilon \to 0}\int_{\mathbb{R}} \psi^{\varepsilon} d(\mu_1 - \mu_2) = \int_{\mathbb{R}}{\psi} d\sigma^+ - \int_{\mathbb{R}}{\psi} d\sigma^- = \int_{P \cap B(0,R)} d\sigma^+ > 0. 
\end{equation*}
Hence, for $\varepsilon$ small enough we have $\int_{\mathbb{R}} \psi^{\varepsilon} d(\mu_1 - \mu_2)>0$, which means that $\rho(\mu_1,\mu_2)>0$. 
\end{proof}
\begin{corollary}
Norm distance, 1-Wassertein distance and bounded Lipschitz distance are metrics on $\mathcal{M}(\mathbb{R})$. 
\end{corollary}

The choice of metric, equivalent to the choice of TFS, is dictated by properties of the system that is being modelled. In case of physical or biological models $\rho(\mu_1,\mu_2)$ should reflect the energy necessary to transform system represented by $\mu_1$ into system represented by $\mu_2$. Large value of $\rho(\mu_1,\mu_2)$ means that transformation from $\mu_1$ to $\mu_2$ is energetically expensive. Conversely, small value of $\rho(\mu_1,\mu_2)$ means that configurations $\mu_1$ and $\mu_2$ are energetically close to each other. Let us consider a generic example.
\begin{example}
\label{Ex_class}
Let $\mu_1 = \delta_0$ and $\mu_2 = \delta_{\varepsilon}$, where $0 < \varepsilon \ll 1$. Then $$\int_{\mathbb{R}} \psi d(\mu_1 - \mu_2) = \psi(0) - \psi(\varepsilon).$$ Taking $\psi(x) = \bold{1}_{(-\infty,0]}(x) - \bold{1}_{(0,\infty)}(x)$, where $\bold{1}_A(x)$ equals $1$ if $x\in A$ and $0$ otherwise, we obtain that $\|\mu_1-\mu_2\|=2.$ On the other hand, $$\rho_F(\mu_1,\mu_2) = \rho_W(\mu_1,\mu_2) = \varepsilon,$$ which follows by observing that $Lip(\psi)\le 1$ implies $\psi(0) - \psi(\varepsilon) \le \varepsilon$ and taking test function 
$$\psi(x)= \bold{1}_{(-\infty,0]}(x) + (1-x)\bold{1}_{(0,2)}(x) + (-1)\bold{1}_{[2,\infty)}(x).$$ 
\end{example}
Example \ref{Ex_class} shows that in $\| \cdot \|$ every pair of different states $x$ is distant from one another. Contrarily, in $\rho_{F}$ and $\rho_W$ the distance of states represented by close enough points $x_1$ and $x_2$ is equal to $|x_1 - x_2|$.

\quad\\ \quad \\
\noindent{\bf{Measure-Transmission metric}}\\
The \emph{Measure-Transmission metric} $\rho_{MT}$ on $\mathcal{M}(\mathbb{R})$ is a combination of flat metric and norm distance. It is well adapted to cell differentiation models, which are considered in this paper. 

To motivate its choice, let $x_0 < x_1 < \dots <x_N$ be points in $\mathbb{R}$, which correspond to discrete states of system  \eqref{eq_munonlinold}-\eqref{eq_icnonlinold}. We demand $\rho_{MT}(\delta_{x_i}, \delta_{x_i +\varepsilon})$ to be large for $0 < \varepsilon \ll 1$ and $\rho_{MT}(\delta_{x_i}, \delta_{x_i -\varepsilon})$ to be small for $0 < \varepsilon \ll 1$. This can be obtained by taking a TFS, which is composed of functions which are Lipschitz-continuous on intervals $(x_{i-1},x_{i}]$, see Figure \ref{Fig_TFS}. 

\begin{figure}[h!]
\centering
\includegraphics[width=8cm]{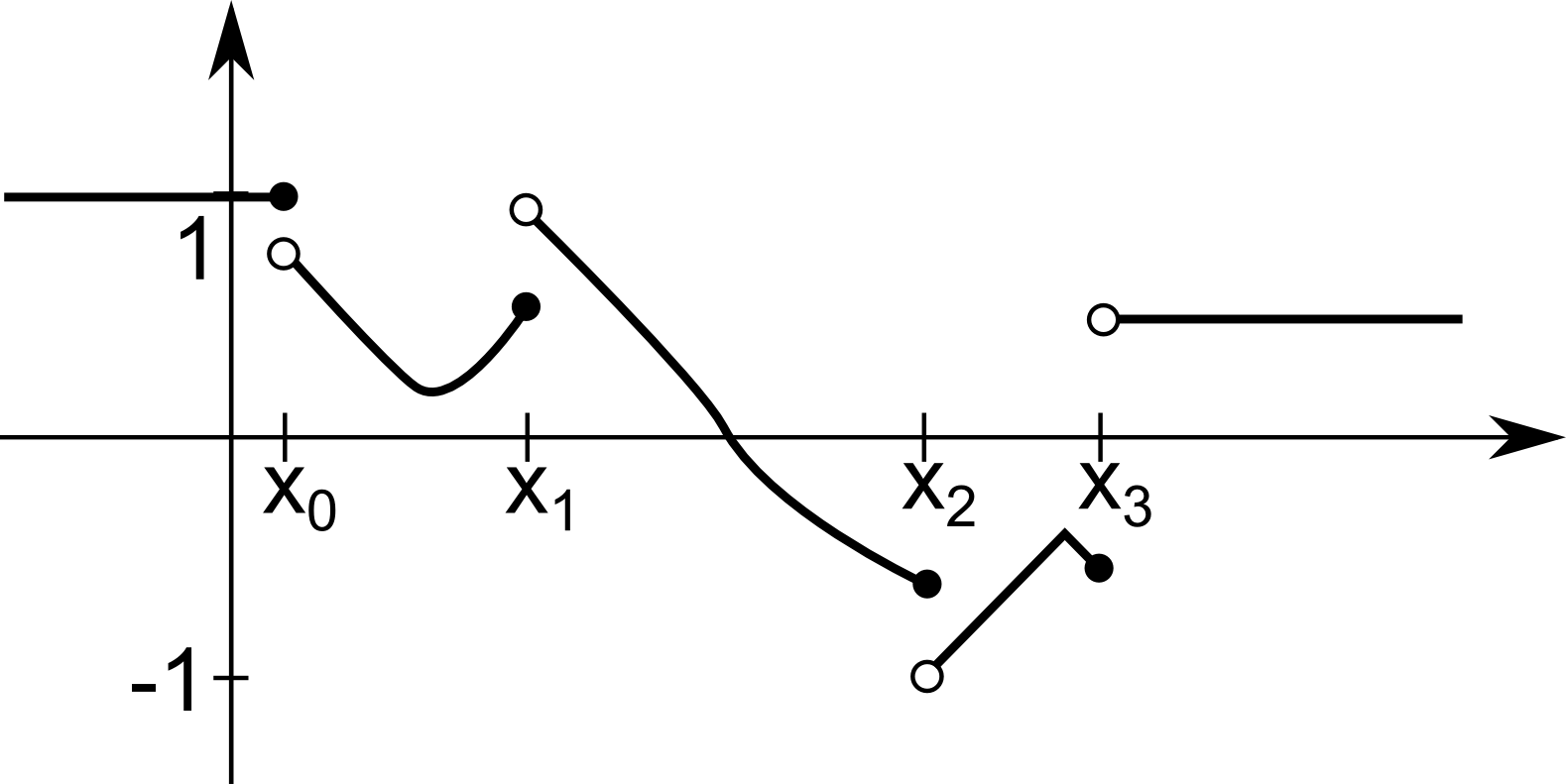}
\caption{Exemplary test function belonging to the space $B_{MT}(\mathbb{R})$ of test functions for the measure-transmission metric. The function is bounded by $1$ and Lipschitz-continuous with constant $1$ on intervals $(x_{i-1},x_{i}]$.}
\label{Fig_TFS}
\end{figure}

The space, the norm in it and the unit ball are defined as follows.

\begin{definition}[Test function space for $\rho_{MT}$]\
Let $x_0 < x_1 < \dots <x_N$ be arbitrary points in $\mathbb{R}$. We define:
\begin{eqnarray*}
W^b_{MT}(\mathbb{R}) &:=& \big\{\psi \in \mathcal{B}(\mathbb{R}):  \sup|\psi| < \infty,  \|\psi\lfloor_{(-\infty,x_0]}\|_{\rm Lip} < \infty, \|\psi\lfloor_{(x_0,x_1]}\|_{\rm Lip}< \infty, \dots ,\\
&& \qquad \qquad\qquad \qquad\qquad \qquad \|\psi\lfloor_{(x_{N-1},x_{N}]}\|_{\rm Lip}< \infty,  \|\psi\lfloor_{(x_{N},+\infty)}\|_{\rm Lip}< \infty\ \big\}.\\
\|\psi\|_{W^b_{MT}} &:=& \max \left( \sup|\psi|,  
\|\psi\lfloor_{(-\infty,x_0]}\|_{\rm Lip}, \|\psi\lfloor_{(x_0,x_1]}\|_{\rm Lip}, \dots ,\|\psi\lfloor_{(x_{N-1},x_{N}]}\|_{\rm Lip},\|\psi\lfloor_{(x_{N},+\infty)}\|_{\rm Lip} \right).\\
B_{MT}(\mathbb{R}) &:=& \{\psi \in W^b_{MT}: \|\psi\|_{W^b_{MT}} \le 1 \}. 
\end{eqnarray*}
\end{definition}
$W^b_{MT}$ equipped with norm $\|\cdot\|_{W^b_{MT}}$ is a Banach space as a direct product of a finite number of Banach spaces of Lipschitz continuous functions on $(x_{i-1},x_i]$ for $i\in \{0,\dots,N+1\}$, where $x_{-1}:=-\infty, x_{N+1}:=\infty$. 

\begin{definition}[Measure-transmission metric]
\label{def_rhoMT}
Let $\mu_1, \mu_2$ be finite Radon measures on $\mathbb{R}$. We define the \emph{measure-transmission metric} by
\begin{equation*}
\rho_{MT}(\mu_1,\mu_2) :=  \sup_{\psi \in B_{MT}(\mathbb{R})} \int_{\mathbb{R}} \psi d(\mu_1 - \mu_2).
\end{equation*}
\end{definition}
\begin{proposition}
$\rho_{MT}$ is a metric.
\end{proposition}
\begin{proof}
Follows by Proposition \ref{Prop_ismetric}. 
\end{proof}
\begin{example}
$\rho_{MT}(\delta_{x_1},\delta_{x_1 +\varepsilon}) = 2$, whereas $\rho_{MT}(\delta_{x_1},\delta_{x_1-\varepsilon})=\varepsilon$.
\end{example}
\begin{proof}
In case of $\rho_{MT}(\delta_{x_1},\delta_{x_1+\varepsilon})$ the supremum from Definition \ref{def_rhoMT} is realized by $$\psi = 
\bold{1}_{(-\infty,x_1]}(x) - \bold{1}_{(x_1,\infty)}(x).$$ In case of $\rho_{MT}(\delta_{x_1},\delta_{x_1-\varepsilon})$ the supremum is realized by $$\psi = (-1)\bold{1}_{(-\infty,x_1-2]}(x) + (x-x_1+1)\bold{1}_{(x_1-2,x_1)} + \bold{1}_{[x_1,\infty)}.$$
Note that we cannot use function $\psi = 
\bold{1}_{(-\infty,x_1)} - \bold{1}_{[x_1,\infty)}$, since it is not left-continuous in $x_1$. 
\end{proof}
In Table \ref{Tab_VarMetrics} we summarize the behaviour of metrics considered in this section in the vicinity of points $x_i$.
\begin{table}[h!]
\centering
\begin{tabular}{|c|c|c|c|c|}
\hline
{\rm \bf Metric} & $\| \cdot \|$  & $\rho_{MT}$ & $\rho_W$ & $\rho_{F}$\\
\hline
{\rm \bf Distance of $\delta_{x_1}$ and $\delta_{x_1+\varepsilon}$} & $2$ & $2$ &$\varepsilon$ & $\varepsilon$\\
\hline 
{\rm \bf Distance of $\delta_{x_1}$ and $\delta_{x_1-\varepsilon}$} & $2$ & $\varepsilon$ & $\varepsilon$ & $\varepsilon$\\
\hline 
\end{tabular}
\caption{Perturbations of $\delta_{x_1}$ calculated in various metrics.}
\label{Tab_VarMetrics}
\end{table}

The measure-transmission metric can be thought of as halfway between $\| \cdot \|$ and $\rho_F$. Namely, it has properties of the flat metric to the left of $x_i$ and of the norm distance to the right of $x_i$, which corresponds to an energy barrier at discrete states $x_i$.

\section{Modified framework of cell differentiation}
\label{Sec_Framework}
The framework for modelling cell differentiation processes, introduced in \cite{OUR} and briefly presented in Section \ref{Sec_intro}, is given by the following equations:
\begin{eqnarray}
\partial_t \mu(t) + \partial_x (g_1(v(t))\bold{1}_{x \neq x_i}(x) \mu(t)) &=& p(v(t),x) \mu(t),\label{eq_munonlin}\\
g_1(v(t))  \frac {D\mu(t)}{D\mathcal{L}^1} (x_i^+) &=& c_i(v(t)) \int_{\{x_i\}} d\mu(t), \label{eq_bcnonlin} \quad\quad i = 0,\dots,N\\
\mu(0) &=& \mu_0,\label{eq_icnonlin}
\end{eqnarray}
where $t\in \mathbb R^+$ and  $x\in \mathbb R$.   $x_0 < x_1 < \dots < x_N$ is a finite collection of points in $\mathbb{R}$,
$\bold{1}_{x \neq x_i}$ is equal $1$ if $x \in (x_0,x_1) \cup (x_1,x_2) \cup \dots \cup (x_{N-1},x_N)$ and $0$ otherwise. $\frac {D\mu}{D\mathcal{L}^1}$ denotes the density of measure $\mu$ with respect to the one-dimensional Lebesgue measure and $v(t) := \int_{\{x_N\}} d\mu(t)$ denotes the mass of point $x_N$. The initial datum $\mu_0$ is a Radon measure supported on the interval $[x_0,x_N]$.
The assumptions on coefficients are following.
\begin{assumptions}[see \cite{OUR}, Assumptions 3.2] \label{Assumptions}
\begin{itemize}
\item[(i)]$g_1(v) \in Lip^b(\mathbb R)$, and $g_1>0$,
\item[(ii)]$p = p(v(t),x)=p_1(v(t))p_2(x)$,
\item[(iii)]$p_1(v) \in Lip^b(\mathbb R)$,
\item[(iv)]$p_2(x) \in \mathcal{B}^b(\mathbb{R})$, $p_2(x)=0$ for $x \in \mathbb{R}\backslash [x_0,x_N]$ and $p_2$ restricted to $(x_{i-1},x_i)$ is Lipschitz continuous for every $i \in \{1,\dots,N\},$ 
\item[(v)]$c_i=c_i(v) \in Lip^b(\mathbb R), \quad i=0,1, \ldots ,N$,
\item[(vi)]$c_i\ge0, \quad i=0,1,\ldots,N,$
\item[(vii)]$c_N=0$.
\end{itemize}
\end{assumptions}
Above, $\mathcal{B}^b(\mathbb{R})$ stands for the space of bounded Borel functions on $\mathbb{R}$ and $Lip^b(\mathbb{R})$ for the space of bounded Lipschitz functions on $\mathbb{R}$.
The solutions are defined as follows. 

\begin{definition}[$\rho$-measure-transmission solution, see Definition 3.3 from \cite{OUR}] \label{def_distrsol}
Let $\mu_0$ be a Radon measure supported on $[x_0,x_N]$. A measure-valued function $\mu \in C([0,\infty),({\cal M},\rho))$ with 
$\int_{\{x_N\}}d\mu(t) \in BV_{loc}([0,\infty))$ is  called a \emph {$\rho$-measure-transmission solution} of problem \eqref{eq_munonlin}--\eqref{eq_icnonlin}, if 
\begin{enumerate}[i)]
\item  for every $\phi \in C_c^{\infty}([0,\infty)\times \mathbb{R})$
\begin{eqnarray}\nonumber
- \int_{ \mathbb{R}^+} \int_{ \mathbb{R}} \partial_t \phi(t,x) d\mu(t)(x)dt -  \int_{ \mathbb{R}^+} \int_{ \mathbb{R}} g_1(v(t))\bold{1}_{x\neq x_i}(x)\partial_x \phi(t,x)  d\mu(t)(x)dt \\ \label{WeakQuestion}= \int_{ \mathbb{R}^+} \int_{ \mathbb{R}}  p_1(v(t))p_2(x)\phi(t,x)d\mu(t)(x)dt + \int_{ \mathbb{R}} \phi(0,x)d\mu_0(x),
\end{eqnarray}
\item for every $t^* > 0$ there exists $\varepsilon(t^*)$ such that for every $t>t^*$ measure $\mu(t)$ is absolutely continuous with respect to the Lebesgue measure $\mathcal{L}^1$ for $x \in (x_i, x_i+\varepsilon)$ and for $\mathcal{L}^1$ a.e. $t \in (0,\infty)$ 
\begin{equation*}
\lim_{x \to x_i^+} g_1(v(t)) \frac {D\mu(t)}{D\mathcal{L}^1} (x) = c_i(v(t))\int_{x_i} d\mu(t),
\end{equation*}
\item for every $i = 0,1, \dots, N$ we have $\int_{\{x_i\}} d\mu(t) \to \int_{\{x_i\}} d\mu(0)$ as $t \to 0$.
\end{enumerate}
\begin{eqnarray*}
\end{eqnarray*}
Above, $BV_{loc}([0,\infty))$ is the space of right-continuous functions, which are of bounded variation on every finite subinterval of $[0,\infty)$ (we refer e.g. to \cite{Lojasiewicz} for definition and properties of $BV$ functions).
\end{definition}
The following theorem summarizes the analytical content of \cite{OUR}.
\begin{theorem}[Existence and uniqueness of $\rho_F$-measure-transmission solutions]
\label{ThmNonlinF}
For every Radon measure $\mu_0\in \mathcal{M}(\mathbb{R})$ such that $supp(\mu_0) \subset [x_0,x_N]$, there exists a unique measure-transmission solution of problem  \eqref{eq_munonlin}--\eqref{eq_icnonlin} in the sense of Definition \ref{def_distrsol} with $\rho=\rho_F$.
\end{theorem}
As observed in Example \ref{Ex_stabnonstab}, in case $\rho=\rho_F$ the solutions lack continuity with respect to perturbation of the initial condition. The choice of metric $\rho=\rho_{MT}$ fixes this defect. The well-posedness results in the new setting are contained in Theorem \ref{ThmNonlin} (existence and uniqueness) and Theorem \ref{StabilityThm} (stability).  
\begin{theorem}[Existence and uniqueness of $\rho_{MT}$-measure-transmission solutions]
\label{ThmNonlin}
For every Radon measure $\mu_0\in \mathcal{M}(\mathbb{R})$ such that $supp(\mu_0) \subset [x_0,x_N]$, there exists a unique measure-transmission solution of problem  \eqref{eq_munonlin}--\eqref{eq_icnonlin} in the sense of Definition \ref{def_distrsol} with $\rho=\rho_{MT}$.
\end{theorem}
\begin{proof}
Observe that $C([0,\infty), (\mathcal{M},\rho_{MT})) \subset C([0,\infty), (\mathcal{M},\rho_{F})).$ Thus, uniqueness follows from Theorem \ref{ThmNonlinF}. Existence of solutions is a consequence of observation that the proof of Lemma 4.9 from \cite{OUR} carries over with no change to the case of $\rho_{MT}$. Thus, solutions defined explicitly by formulas (17)-(21) in \cite{OUR} belong not only to ${\rm Lip}_{loc}([0,T], (\mathcal{M},\rho_{F}))$ but also to ${\rm Lip}_{loc}([0,T], (\mathcal{M},\rho_{MT}))$ and hence to $C([0,\infty), (\mathcal{M},\rho_{F}))$. Change of the time variable in \cite[Definition 6.1]{OUR} preserves this regularity. Thus, solutions constructed in \cite{OUR} belong in fact to $C([0,\infty), (\mathcal{M},\rho_{MT}))$, which concludes the proof.
\end{proof}

\begin{remark}
\rm
\begin{enumerate}[i)]
\item It is possible to adopt a more general approach to existence and uniqueness of solutions based on the superposition solution technique, see \cite{MyPHD}. 
\item The assumptions of Definition \ref{def_distrsol} can be relaxed. This leads to additional technical difficulties and is fully treated in \cite{MyPHD}, see also Remark \ref{Rem16}.
\end{enumerate}
\end{remark}

\noindent Now, we formulate our main result. Let
\begin{itemize}
\item $\sup(c) := \max_{i \in \{0, \dots N\}} \sup_{v \in \mathbb{R}} |c_i(v)|$,
\item $\sup(g_1):= \sup_{v \in \mathbb{R}} g_1(v)$,
\item $\min(g_1):= \min_{k =1,2} \inf_{t \in [0,\infty)}g_1(v_k(t)),$ where $v_k(t) = \int_{\{x_N\}} d\mu_k(t)$,
\item ${\rm Lip}(g_1)$ be the Lipschitz constant of $g_1$,
\item ${\rm Lip}(c):= \max_{i \in 0,\dots,N} {\rm Lip}(c_i)$, where ${\rm Lip}(c_i)$ are the Lipschitz constants of functions $c_i$,
\item $TV(\mu):=\int_{\mathbb{R}} d\mu$ be the total variation of $\mu$.
\end{itemize}
Then the following stability theorem holds.
\begin{theorem}[Stability of $\rho_{MT}$-measure-transmission solutions in case $p=0$]
\label{StabilityThm}
Let $\mu_1(t)$ and $\mu_2(t)$ be two $\rho_{MT}$-measure-transmission solutions of system \eqref{eq_munonlin}-\eqref{eq_icnonlin} with $p\equiv 0$, corresponding to initial conditions $\mu_1(0)$ and $\mu_2(0)$, respectively. There exist constants $\alpha,\beta$, dependent only on $\sup(c)$, $\sup(g_1)$, $\min(g_1)$,  ${\rm Lip}(g_1)$, ${\rm Lip}(c)$, $TV(\mu_1(0))$, $TV(\mu_2(0))$  such that
\begin{equation}
\label{Eq_stabestimate}
\rho_{MT}(\mu_1(t),\mu_2(t)) \le e^{\alpha \left\lceil \frac t \beta \right\rceil} \rho_{MT} (\mu_1(0),\mu_2(0)),
\end{equation}
where $ \left\lceil \frac t \beta \right\rceil$ is the smallest integer greater or equal $\frac {t}{\beta}$.\footnote{In particular, $\lim_{t \rightarrow 0^+} e^{\alpha \left\lceil \frac t \beta \right\rceil} = e^{\alpha}$.}
\end{theorem}
The proof of Theorem \ref{StabilityThm} is presented in Section \ref{Sec_Stabilityp0}.  Note that, for simplicity, we consider only case $p=0$, postponing the full result to further work.

\section{Proof of the stability theorem in case $p=0$}

\label{Sec_Stabilityp0}

In this chapter we prove Theorem \ref{StabilityThm}. We consider, namely, the system of equations
\begin{eqnarray}
\partial_t \mu(t) + \partial_x (g_1(v(t))\bold{1}_{x \neq x_i}(x) \mu(t)) &=& 0,\label{eq_munonlinp0}\\
g_1(v(t))  \frac {D\mu(t)}{D\mathcal{L}^1} (x_i^+) &=& c_i(v(t)) \int_{\{x_i\}} d\mu(t), \label{eq_bcnonlinp0}\ \ \ \ i = 0, \dots ,N\\
\mu(0) &=& \mu_0,\label{eq_icnonlinp0}
\end{eqnarray}
which is a simplification of system \eqref{eq_munonlin}-\eqref{eq_icnonlin} obtained by taking $p=0$.

To prove Theorem \ref{StabilityThm} we take two $\rho_{MT}$-measure-transmission solutions $\mu_1(t),\mu_2(t)$, denote  $v_j(t):=\int_{\{x_N\}} d\mu_j(t)$ for $j\in \{1,2\}$ and  proceed in the following steps:
\begin{enumerate}
\item We prove a 'superposition principle' (see \cite{ambrosio, CrippaPHD}) for system  \eqref{eq_munonlinp0}-\eqref{eq_icnonlinp0}, which allows us to express its solutions as certain combinations over characteristics called \emph{superposition solutions}.
\item We obtain an estimate of $\int_0^T |v_1(t)-v_2(t)|dt$ in terms of $\rho_{MT}(\mu_1(0),\mu_2(0))$ and $\int_U d\mu_1(0), \int_U d\mu_2(0)$, where $U$ is some neighborhood of $x_N$ (Nonlinear Estimate).
\item We obtain an estimate of $\rho_{MT}(\mu_1(t), \mu_2(t))$ for small $t$ in terms of $\int_0^t |v_1(s) - v_2(s)|ds$ and $\rho_{MT}(\mu_1(0),\mu_2(0))$ (Linear Estimate). 
\item We substitute the Nonlinear Estimate into the Linear Estimate to obtain an estimate of $\rho_{MT}(\mu_1(t),\mu_2(t))$ in terms of $\rho_{MT}(\mu_1(0),\mu_2(0))$ for small $t$. 
\item We prolong the estimate to large $t$.  
\end{enumerate}
\begin{remark}
\label{Rem16}
\rm Steps 2-5, presented above, are based solely on the fact that every measure-transmission solution can be represented as superposition solution, i.e. in terms of formulas \eqref{eq_defeta}-\eqref{eq_defeta2}. Thus, estimate \eqref{Eq_stabestimate} holds true for every pair of measure-valued functions $\mu_1,\mu_2 : [0,T) \to \mathcal{M}(\mathbb{R})$, which satisfy \eqref{eq_defeta}-\eqref{eq_defeta2}.
In particular, if the definition of measure-transmission solutions is modified in a way, which preserves the superposition principle, then stability estimate \eqref{Eq_stabestimate} remains valid.
This comment is motivated by the fact that uniqueness criteria ii)-iii) of Definition \ref{def_distrsol}, introduced in \cite{OUR}, which are an interpretation of the measure-transmission conditions \eqref{eq_bcnonlin}, are somewhat artificial.  More natural uniqueness criteria in the definition of solutions are studied in \cite{MyPHD}, where also, in contrast to \cite{OUR}, \emph{detailed} proofs of existence and uniqueness of measure-transmission solutions  are provided. As noted above, the stability estimate \eqref{Eq_stabestimate} carries over also to that case. 
\end{remark}

\subsection{Superposition principle}
In this section we show that measure-transmission solutions can be represented in terms of characteristics. 
Let, namely,  $T_{max}$, $G$ and $\tau(x_b)$, where $x_b \in \mathbb{R}$, be defined by
\begin{equation}
\label{Eq_defTmax2}
T_{max}:= \frac{\min_{i \in \{1,2,\dots,N\}} |x_{i}-x_{i-1}|} {\sup(g_1)},
\end{equation} 
\begin{equation}
G(t):= \int_0^t g_1(v(s)) ds,
\end{equation}
\begin{equation}
\label{Eq_deftautau}
\tau(x_b) := \inf \{t \in [0,\infty): x_b + G(t) \in  \{x_0,x_1,\dots,x_N\} \}.
\end{equation} 
Let, moreover, $X(x_b,0,r,\cdot)$ be, for $r\ge \tau(x_b)$, an absolutely continuous solution of equation $\dot{x}=\bold{1}_{x \neq x_i} g_1(v)$ given by formula
\begin{equation}
\label{Eq_defX}
X(x_b,0,r,t) := \begin{cases}
x_b + G(t)   &\mbox{ for } t \le \tau(x_b),\\
x_b + G(\tau(x_b))  &\mbox{ for } \tau(x_b) < t \le r, \\
x_b + G(\tau(x_b)) + G(t) - G(r)  &\mbox{ for } r < t \le T.
\end{cases}
\end{equation}
We interpret  $X(x_b,0,r,\cdot)$ as  the unique characteristic generated by $g_1(v)$ with a branching time $r$, see Figure \ref{Fig_charex3}. 
We obtain the following result.
\begin{proposition}[Superposition principle]
\label{Prop_superprinciple}
Let $\mu$ be a $\rho_{MT}$-superposition solution of  \eqref{eq_munonlinp0}-\eqref{eq_icnonlinp0}. Then for every bounded Borel function $\phi \in \mathcal{B}^b(\mathbb{R})$ and $T<T_{max}$ with $T_{max}$ given by \eqref{Eq_defTmax2} we have
\begin{equation}
\label{eq_defeta}
\int_{\mathbb{R}} \phi d\mu(T) = \int_{\mathbb{R}} \left( \int_{[0,T]} \phi(X(x_b,0,r,T)) d\eta_{x_b}(r) \right) d\mu(0)(x_b),
\end{equation}
where 
\begin{equation}
d\eta_{x_b}(r) := \begin{cases} 
e^{-\int_{\tau(x_b)}^T c_{\lambda}(v(s))ds} \delta_T(dr) + c_{\lambda}(v(r))e^{-\int_{\tau(x_b)}^r c_{\lambda}(v(s))ds} \bold{1}_{[\tau(x_b),T]}(r)dr\\
 \mbox{\quad \quad \quad \quad if } x_{\lambda-1} < x_b \le x_\lambda \mbox{ for some } \lambda \in \{1,\dots,N-1\} \mbox{ and } \tau(x_b)\in [0,T],\\
\delta_T(dr) \mbox{ otherwise.}
\end{cases}
\label{eq_defeta2}
\end{equation}
\end{proposition}

\begin{remark}
\rm
By the general superposition principle for continuity equation, see \cite[Theorem 6.2.2]{CrippaPHD}, we obtain that there \emph{exist} measures $\eta_{x_b}$ such that \eqref{eq_defeta} holds. Proposition \ref{Prop_superprinciple} provides, in addition, an explicit formula for $\eta_{x_b}$, which is useful in subsequent computations. 
\end{remark}

\begin{proof}[Proof of Proposition \ref{Prop_superprinciple}]
It is a simple calculation that $\eta_{x_b}$ is a probability measure for every $x_b$. 
Thus, it remains to show that the left-hand side (LHS) of \eqref{eq_defeta}, calculated using formulas (17)-(20) and Definition 6.1 from \cite{OUR}, equals the right-hand side (RHS) of \eqref{eq_defeta} calculated explicitly using formula \eqref{eq_defeta2}. We proceed in two steps: $g_1 \equiv 1$ and arbitrary $g_1$. In the following, for fixed solution $\mu$, we denote $c_i(s):=c_i(v(s))$, $i \in \{0,1,\dots,N\}$ and $g_1(s):=g_1(v(s))$. Functions $h_i$ are defined by formula (18) from \cite{OUR} and by 'characteristics end in A' we mean that $X(x_b,0,r,T) \in A$.
\\ \, \\
\noindent{\bf Step 1} ($g_1 \equiv 1$). We begin with three special cases. 
\begin{enumerate}[(a)]
\item $\phi = \bold{1}_A$ with $A \subset (x_{i-1}+T,x_i)$ for some $i \in \{0,1,\dots,N\}$. Then characteristics ending in $A$ have the shape as in the left panel of Figure \ref{Fig_charex3}. We obtain
\begin{figure}[htbp] 
\begin{center}
\includegraphics[width=10cm]{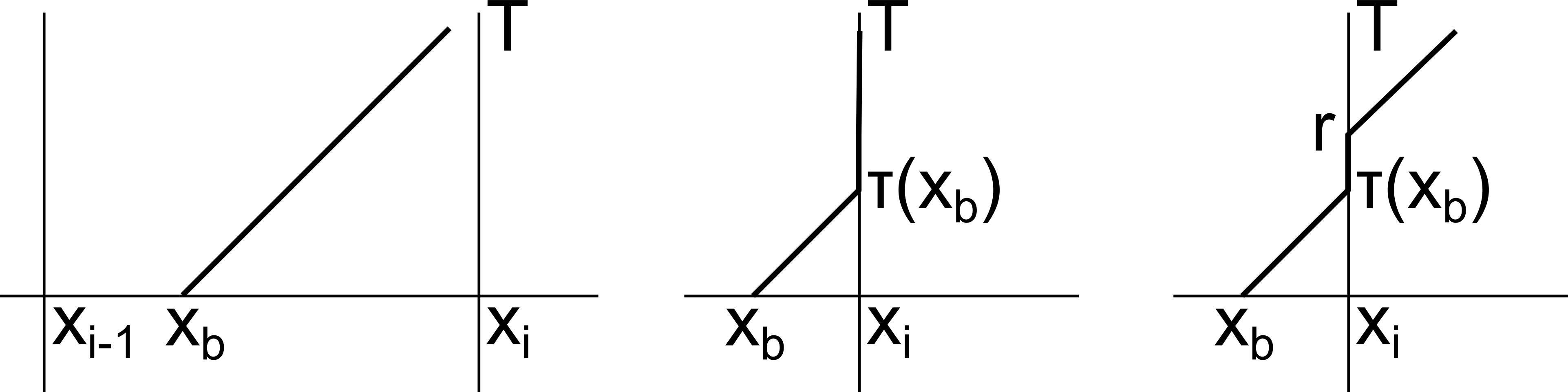}
\caption{Three types of characteristics of system  \eqref{eq_munonlinp0}-\eqref{eq_icnonlinp0} for $g_1 \equiv 1$. In the first case $\tau(x_b)> T$ (left panel). In the second case, $0 \le \tau(x_b) \le T$ and $r \ge T$ (middle panel). In the last case $0 \le \tau(x_b) \le r <T$ (right panel).}
 \label{Fig_charex3}
\end{center}
\end{figure}
\begin{eqnarray*}
RHS &=& \int_{\mathbb{R}} \left(\int_{[0,T]} \bold{1}_{A}(X(x_b,0,r,T))\delta_T(dr) \right) d\mu(0)(x_b)=\int_{\mathbb{R}}  \bold{1}_{A}(x_b+T) d\mu(0)(x_b) =  \int_{A - T}  d\mu(0)(x_b),\\
LHS &=& \mu(T)(A) = \int_{A - T}  d\mu(0)(x_b),
\end{eqnarray*}
where $A-T=\{x: x+T \in A \}$ and we used formula (19) from \cite{OUR} to calculate LHS.
\item $\phi = \bold{1}_A$ with $A={\{x_i\}}$ for some $i \in \{0,\dots,N\}$. Characteristics ending in $A$ have the shape as in the middle panel of Figure \ref{Fig_charex3}. We obtain
\begin{eqnarray*}
RHS &=& \int_{[x_i - T, x_i]} \Bigg(\int_{[\tau(x_b),T]} \bold{1}_{\{x_i\}} (X(x_b,0,r,T))  e^{-\int_{\tau(x_b)}^T c_i(s)ds} \delta_T (dr) 
  \Bigg)  d\mu(0)(x_b) \\
&=& \int_{[x_i - T, x_i]} \left( \bold{1}_{\{x_i\}} (X(x_b,0,T,T))  e^{-\int_{\tau(x_b)}^T c_i(s)ds} \right)   d\mu(0)(x_b), \\
LHS &=& \mu(T)(\{x_i\}) = e^{-\int_0^T c_i(s)ds}\int_{\{x_i\}} d\mu(0) + \int_{(0,T]} h_i(dr) e^{-\int_r^T c_i(s)ds}\\
&=& e^{-\int_0^T c_i(s)ds}\int_{\{x_i\}} d\mu(0) + \int_{[x_i-T,x_i)} e^{-\int_{\tau(x_b)}^T c_i(s)ds} d\mu(0)(x_b)\\
&=& \int_{[x_i-T,x_i]} e^{-\int_{\tau(x_b)}^T c_i(s)ds} d\mu(0)(x_b),
\end{eqnarray*}
where we used formulas (18) and (20) from \cite{OUR} to calculate LHS. 
\item $\phi=\bold{1}_A$ with $A \subset (x_i,x_{i}+T]$ for some $i \in \{0,1,\dots,N\}$. Here, the characteristics assume the shape depicted in the right panel of Figure \ref{Fig_charex3}. As a result, 
\begin{eqnarray*}
RHS &=& \int_{[x_i - T, x_i]} \Bigg(\int_{[\tau(x_b),T]}  \bold{1}_{A} (X(x_b,0,r,T))  c_i(r) e^{- \int_{\tau(x_b)}^r c_i(s)ds}  dr \Bigg)  d\mu(0)(x_b),\\
LHS &=& \int_{T + x_i - A} f_i(r)dr = \int_{T + x_i - A}  c_i(r) \left(\int_{\{x_i\}} d\mu(r)(x_b)\right) dr \\
&=& \int_{T+x_i-A} c_i(r) \left[e^{-\int_0^r c_i(s)ds} \int_{\{x_i\}} d\mu(0) + \int_{(0,r]} h_i(d\tau) e^{-\int_{\tau}^r c_i(s)ds} \right] dr \\
&=& \int_{[x_i-T,x_i]} \int_{\tau(x_b)}^T \bold{1}_{A} (x_i+T-r) e^{-\int_{\tau(x_b)}^r c_i(s)ds} c_i(r) dr d\mu(0)(x_b)\\
&=& \int_{[x_i - T,x_i]} \int_{\tau(x_b)}^T c_i(r) e^{- \int_{\tau(x_b)}^r c_i(s)ds}\bold{1}_{A} (X(x_b,0,r,T)) dr d\mu(0)(x_b),
\end{eqnarray*}
where $T + x_i - A := \{t: x_i + T- t \in A\}$ and we used in turn formulas (19), (18), (20) from \cite{OUR} as well as the Fubini theorem to compute LHS.  
\end{enumerate}
We observe that in every case $RHS = LHS$. 
Since functions of the form $(a),(b),(c)$ generate the whole set of Borel-measureable functions on $\mathbb{R}$, we conclude. \\

\noindent{\bf Step 2} (arbitrary $g_1$). We use \cite[Definition 6.1]{OUR} and handle similarly as in the proof of \cite[Theorem 6.2]{OUR}.
Namely, we define 
\begin{eqnarray*}
\tilde{t}(t) &:=& \int_0^t g_1(s)ds, 
\quad d\tilde{t} := g_1(t)dt,\\
\tilde{c}_i(\tilde{t}) &:=&  \frac {c_i(t(\tilde{t}))}{g_1(t(\tilde{t}))}, \quad i = 0,1, \dots, N,\\
\tilde{X}(x,0,\tilde{r}(r),\tilde{s}(s)) &:=& X(x,0,r,s),\\
\tilde{\mu}(\tilde{t}(t)) &:=& \mu(t).
\end{eqnarray*}
Due to this transformation, $\tilde{\mu}$ satisfies equation \eqref{eq_munonlinp0} with velocity $\tilde{g}_1 \equiv 1$. Thus, using Step 1, we can write
\begin{eqnarray*}
\int_{\mathbb{R}} \phi d\mu(T) = \int_{\mathbb{R}} \phi d\tilde{\mu} (\tilde{T}) = 
\int_{\mathbb{R}} \left(\int_{[0,\tilde{T}]}  \phi(\tilde{X}(x_b,0,\tilde{r},\tilde{T})) d\tilde{\eta}_{x_b} (\tilde{r}) \right) d{\tilde{\mu}}(0)(x_b).
\end{eqnarray*}
Now, we transform the inner integral, using the change of variables defined above. There are two cases, depending on the value of parameter $x_b$. 
\begin{itemize}
\item $\tilde{\eta}_{x_b} = \delta_{\tilde{T}} (d\tilde{r})$. Then
\begin{eqnarray*}
\int_{[0,\tilde{T}]}  \phi(\tilde{X}(x_b,0,\tilde{r},\tilde{T})) d\tilde{\eta}_{x_b} (\tilde{r}) &=&  \phi(X(x_b,0,T,T))= \int_{[0,T]}  \phi(X(x_b,0,r,T) \delta_T (dr).
\end{eqnarray*}
\item 
$\tilde{\eta}_{x_b} =  e^{-\int_{\tilde{\tau}(x_b)}^{\tilde{T}} \tilde{c}_i(\tilde{s})d\tilde{s}} \delta_{\tilde{T}}(d\tilde{r}) + \tilde{c}_i(\tilde{r}) e^{- \int_{\tilde{\tau}(x_b)}^{\tilde{r}} \tilde{c}_i(\tilde{s}) d\tilde{s}} \bold{1}_{[\tilde{\tau}(x_b),\tilde{T}]}(\tilde{r})d\tilde{r}$. Then
\begin{eqnarray*}
&&\int_{[0,\tilde{T}]}  \phi(\tilde{X}(x_b,0,\tilde{r},\tilde{T})) d\tilde{\eta}_{x_b} (\tilde{r}) \\
&=&\int_{[0,\tilde{T}]}  \phi(X(x_b,0,r,T)) \left[ e^{-\int_{\tilde{\tau}(x_b)}^{\tilde{T}} \tilde{c}_i(\tilde{s})d\tilde{s}} \delta_{\tilde{T}}(d\tilde{r}) + \tilde{c}_i(\tilde{r}) e^{- \int_{\tilde{\tau}(x_b)}^{\tilde{r}} \tilde{c}_i(\tilde{s}) d\tilde{s}} \bold{1}_{[\tilde{\tau}(x_b),\tilde{T}]}(\tilde{r})d\tilde{r} \right] \\
&=&\int_{[0,T]} \phi(X(x_b,0,r,T)) \left[e^{-\int_{\tau(x_b)}^T c_i(s)ds} \delta_T(dr) + c_i(r) e^{-\int_{\tau(x_b)}^r c_i(s)ds} \bold{1}_{[\tau(x_b),T]} (r)dr \right].
\end{eqnarray*}
\end{itemize}
Hence,
\begin{eqnarray*}
\int_{\mathbb{R}} \phi d\mu(T) = \int_{\mathbb{R}} \phi d\tilde{\mu} (\tilde{T}) = \int_{\mathbb{R}} \left(\int_{[0,\tilde{T}]} \phi(\tilde{X}(x_b,0,\tilde{r},\tilde{T})) d\tilde{\eta}_{x_b} (\tilde{r}) \right) d{\tilde{\mu}}(0)(x_b)\\
= \int_{\mathbb{R}} \left(\int_{[0,{T}]} \phi({X}(x_b,0,{r},{T})) d{\eta}_{x_b} ({r}) \right) d{\mu}(0)(x_b).
\end{eqnarray*}

\end{proof}

\begin{remark}
A similar calculation, omitted here for simplicity, allows us to prove that for every $\rho_{MT}$-measure-transmission solution of system \eqref{eq_munonlin}-\eqref{eq_icnonlin} and for every $\phi \in \mathcal{B} (\mathbb{R})$ and $t \in [0,T]$, $T<T_{max}$, we have 
\begin{equation}
\int_{\mathbb{R}} \phi d\mu(t) = \int_{\mathbb{R}} \left( \int_{[0,T]}  e^{\int_0^t p(s,X(x_b,0,r,s))ds}\phi(X(x_b,0,r,t)) d\eta_{x_b}(r) \right) d\mu(0)(x_b),
\end{equation}
where $\eta_{x_b}$ is defined by \eqref{eq_defeta2}.
\end{remark}

\subsection{Nonlinear estimate}
\label{Sec_Nonlinear_Estimate_p0}
Our goal here is to estimate $\int_0^T |v_1(t)-v_2(t)|dt$ in terms of $\rho_{MT} (\mu_1(0),\mu_2(0))$ where $T<T_{max}$ and $T_{max}$ is given by \eqref{Eq_defTmax2}.
To this end, we observe that by Proposition \ref{Prop_superprinciple} $v_j$ can be expressed by
\begin{equation}
\label{eq_vj}
v_j(t) = \int_{[x_N - \int_0^t g_1(v_j(s))ds,x_N]} d\mu_j(0) = \int_{[x_N - G_j(t),x_N]} d\mu_j(0),
\end{equation}
where
\begin{equation}
\label{Den_Gj}
G_j(t):=\int_0^t g_1(v_j(s))ds,
\end{equation}
and use the fact that for $p=0$
\begin{equation}
\label{Cor_ming1}
\min(g_1):= \min_{j \in \{1,2\}} \inf_{t \in [0,\infty)} g_1(v_j(t)) > 0
\end{equation}
due to boundedness of $v_j$ and continuity as well as positivity of $g_1$. 

Denote $\min(G_j):=\min(G_1,G_2)$ and $\max(G_j):=\max(G_1,G_2)$. Using \eqref{eq_vj} we obtain
\begin{eqnarray*}
\int_0^T |v_1(t) - v_2(t)| dt = \int_0^T \left| \int_{[x_N - G_1(t),x_N]} d\mu_1(0) - \int_{[x_N - G_2(t),x_N]} d\mu_2(0) \right|dt &\le& \\
\int_0^T \left| \int_{[x_N - \min(G_j),x_N]} d(\mu_1(0) - \mu_2(0)) \right| dt + \int_0^T \int_{[x_N - \max(G_j), x_N - \min(G_j))} d(\mu_1(0) + \mu_2(0))dt &=& \\ I_1 + I_2.
\end{eqnarray*}
Let 
\begin{eqnarray}
B &:=& \left\{t: \int_{[x_N - \min(G_j)(t),x_N]} d(\mu_1(0)-\mu_2(0)) \ge 0\right\}, \nonumber\\
\tau_j(x_b) &:=& \sup \{t>0: x_b+G_j(t)<x_N\}, \label{def_tauiibis}\\
\tau_{min}(x_b) &:=& \min(\tau_1(x_b),\tau_2(x_b)). \nonumber
\end{eqnarray}
Then, by the Fubini theorem (see Figure \ref{Fig_Fubini1}), $I_1$ is equal to
\begin{eqnarray*}
&&\int_0^T (\bold{1}_B - \bold{1}_{\mathbb{R} \backslash B}) \int_{[x_N - \min(G_j),x_N]} d(\mu_1(0) - \mu_2(0))dt \\
&=& \int_{[x_N - \min(G_j(T)),x_N]} \left(\int_{\tau_{\min}(x_b)}^T (\bold{1}_B - \bold{1}_{\mathbb{R} \backslash B}) (t)dt\right) d (\mu_1(0) - \mu_2(0))(x_b)\\ &=& \int_{\mathbb{R}} \chi(x_b) d (\mu_1(0) - \mu_2(0))(x_b).
\end{eqnarray*}

\begin{figure}[htbp]
\centering
\includegraphics[width=5cm]{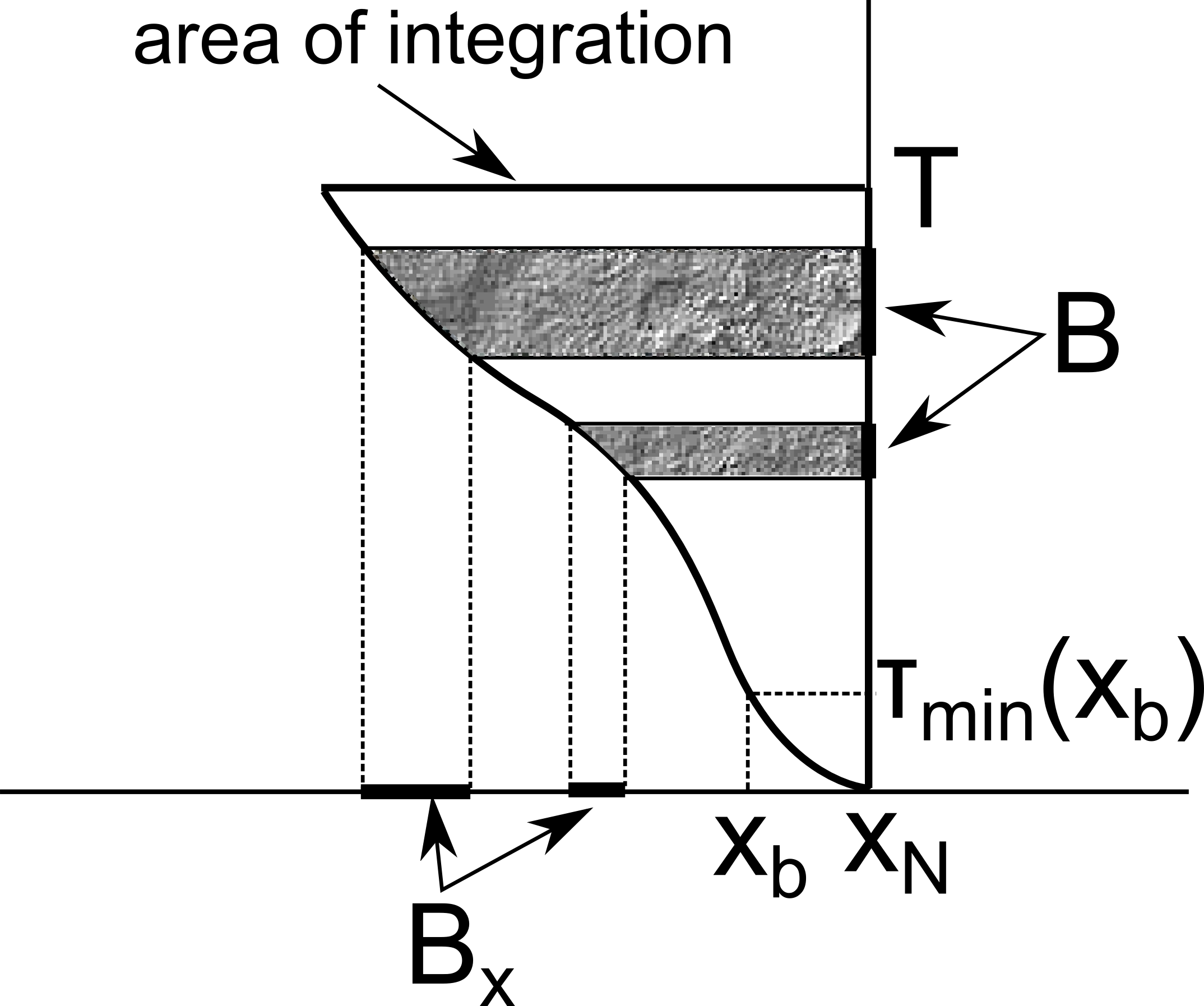}
\caption{The area of integration in $I_1$. The integral can be interpreted as a double integral of function $(\bold{1}_B - \bold{1}_{\mathbb{R} \backslash B})(t)$, which is positive in the shaded region and negative otherwise, with respect to the product measure $(\mu_1(0) - \mu_2(0)) \times dt$. $B_x$ is the set of all $x$ for which $\int_{[x, x_N]} d(\mu_1(0)-\mu_2(0))$ is nonnegative.}
\label{Fig_Fubini1}
\end{figure}
\noindent Function $$\chi(x) := \begin{cases}\int_{\tau_{\min}(x)}^T (\bold{1}_{B} - \bold{1}_{\mathbb{R} \backslash B}) (t)dt &\mbox{ if } x \in [x_N - \min(G_j(T)),x_N], \\ 0 &\mbox{ otherwise} \end{cases}$$
belongs to $W^b_{MT}$. Moreover, $$|\chi(x)| \le T$$ and 
\begin{equation*}
|\chi'(x)| \le |(\bold{1}_B - \bold{1}_{\mathbb{R} \backslash B})| |\tau_{\min}'(x)| \le \frac 1 {\min(g_1)},  
\end{equation*}
where  $\min(g_1)>0$ by \eqref{Cor_ming1}. As a consequence, $\chi = \max\left(\frac {1}{\min(g_1)}, T\right) \chi_1$, where $\chi_1$ belongs to $B_{MT}$.
This leads to conclusion that $$I_1 \le \max\left(\frac {1}{\min(g_1)}, T\right) \rho_{MT} (\mu_1(0),\mu_2(0)). $$ 
Denoting $J_{max}= (x_N - \max(G_1(T),G_2(T)), x_N)$ and $J_{min} = (x_N - \min(G_1(T),G_2(T)),x_N]$ and using the Fubini theorem as well as Proposition \ref{Est_4} we estimate $I_2$ by 
\begin{eqnarray*}
I_2 &\le& \sup_{x \in J_{min}} |\tau_1(x) - \tau_2(x)| (\mu_1(0)(J_{\max}) + \mu_2(0)(J_{\max}))\\
 &\le& (\mu_1(0)(J_{\max}) + \mu_2(0)(J_{\max}))\frac {{\rm Lip}(g_1)}{\min(g_1)} \int_0^T |v_1(t) - v_2(t)|dt.
\end{eqnarray*}
Combining estimates for $I_1$ and $I_2$ we obtain for $T$ small enough
\begin{equation}
\int_0^T |v_1(t) - v_2(t)|dt \le \max\left(\frac {1}{\min(g_1)},T\right) \frac {1}{\left( 1 - \frac {{\rm Lip}(g_1)}{\min(g_1)} (\mu_1(0)(J_{\max})+ \mu_2(0)(J_{max})) \right)} \rho_{MT} (\mu_1(0),\mu_2(0)).
\label{Nonlinear_Estimate}
\end{equation}
Note that the maximum time $T$, up to which estimate \eqref{Nonlinear_Estimate} is valid, strongly depends on $\mu_1(0)$ and $\mu_2(0)$ via $J_{max}$ and cannot be controlled easily. Importantly, however,  $x_N$ does not belong to $J_{\max}$, which will allow us to prolong the stability estimate to arbitrary times, see Section \ref{Sec_LargeTimes}.
\begin{remark}
For $g_1 \equiv 1$ estimate \eqref{Nonlinear_Estimate} turns into
\begin{equation}
\label{Nonlinear_Estimate_g1}
\int_0^T |v_1(t) - v_2(t)|dt \le \max \left(1,T\right) \rho_{MT}(\mu_1(0),\mu_2(0)).
\end{equation}
\end{remark}

\subsection{Linear estimate}
In this section we estimate the quantity $$\rho_{MT}(\mu_1(T),\mu_2(T)) := \sup_{\psi \in B_{MT}} \int_{\mathbb{R}} \psi d(\mu_2(T) - \mu_1(T))$$ for $T<T_{max}$.
The main idea consists in splitting the integral $\int_{\mathbb{R}} \psi d(\mu_2(T) - \mu_1(T))$ into 
\begin{itemize}
\item parts that can be bounded in terms of $\int_0^T |v_1(s) - v_2(s)|ds$ and 
\item parts, which add up to $\int_{\mathbb{R}} \psi^0 d(\mu_2(0) - \mu_1(0))$ for some function $\psi^0 \in  W^b_{MT}$.
\end{itemize}
Then we bound both of them by $C_1(t)\rho_{MT}(\mu_1(0),\mu_2(0))$. 

To achieve this goal, we fix $T<T_{max}$, where $T_{max}$ is given by \eqref{Eq_defTmax2}, and assume without loss of generality (compare Remark \ref{Remark_crosschar}) that $G_1(t)\le G_2(t)$ for $0\le t \le T$, where $G_1,G_2$ are given by \eqref{Den_Gj}. By 
the superposition principle (Proposition \ref{Prop_superprinciple}) we have
\begin{equation}
\label{Eq_defmuit}
\int_{\mathbb{R}} \psi d\mu_j(T) = \int_{\mathbb{R}} \left( \int_{[0,T]} \psi(X_j(x_b,0,r,t)) d\eta^j_{x_b}(r) \right) d\mu_j(0)(x_b),
\end{equation}
where:
\begin{itemize}
\item $j \in \{1,2\}$ enumerates the two solutions,
\item $X_j$ is the characteristic generated by $g_1(v_j)$ (see \eqref{Eq_defX}),
\item
\begin{equation}
\label{Eq_deftautauj}
\tau_j(x_b) := \inf \{t \in [0,\infty): x_b + G_j(t) \in  \{x_0,x_1,\dots,x_N\} \},
\end{equation} 
\item
\begin{equation*}
d\eta^j_{x_b}(r) := \begin{cases} 
e^{-\int_{\tau_j(x_b)}^T c_{\lambda}(v_j(s))ds} \delta_T(dr) + c_{\lambda}(v_j(r))e^{-\int_{\tau_j(x_b)}^r c_{\lambda}(v_j(s))ds} \bold{1}_{[\tau_j(x_b),T]}(r)dr\\
 \mbox{\quad \quad \quad \quad if } x_{\lambda-1} < x_b \le x_\lambda \mbox{ for some } \lambda \in \{1,\dots,N-1\} \mbox{ and } \tau_j(x_b)\in [0,T],\\
\delta_T(dr) \mbox{ otherwise.}
\end{cases}
\end{equation*}
\end{itemize}
\ \newline \ \newline \ \newline \ \newline \ \newline

Using this representation, we split the integral $\int_{\mathbb{R}} \psi d(\mu_2(T) - \mu_1(T))$ into three main components with respect to the starting point of characteristics, $x_b$:
\begin{itemize}
\item Characteristics starting in $(x_{i-1},x_{i}-G_2(T))$ (Fig. \ref{Fig_char2}) -- terms $I_i$,
\item Characteristics starting in $[x_i - G_2(T),x_i - G_1(T))$ (Fig. \ref{Fig_char3}) -- terms $T_i$,
\item Characteristics starting in $[x_i - G_1(T),x_i]$ (Fig. \ref{Fig_char1}) -- terms $D_i$. 
\end{itemize} 
We obtain
\begin{equation*}
\int_{\mathbb{R}} \psi d(\mu_2(T) - \mu_1(T)) = D_0 + (I_1 + T_1 + D_1) + (I_2 + T_2 + D_2) + \dots + (I_N + T_N + D_N),
\end{equation*}
where 
\begin{eqnarray*}
I_i &=&  \int_{(x_{i-1},x_{i}-G_2(T))} \left( H_2(x_b)d\mu_2(x_b) - H_1(x_b)d\mu_1(x_b)\right), \\
T_i &=&  \int_{[x_i - G_2(T),x_i - G_1(T))} \left( H_2(x_b)d\mu_2(x_b) - H_1(x_b)d\mu_1(x_b)\right),\\
D_i &=&  \int_{[x_i - G_1(T),x_i]} \left( H_2(x_b)d\mu_2(x_b) - H_1(x_b)d\mu_1(x_b)\right)\\
\end{eqnarray*}
and we denoted
\begin{equation*}
H_j(x_b)=\left( \int_{[0,T]} \psi(X_j(x_b,0,r,t)) d\eta^j_{x_b}(r) \right) d\mu_j(0)(x_b).
\end{equation*}
Now, we estimate $I_1$, $T_1$ and $D_1$, the calculations for other terms being similar. In the estimates, we further group the characteristics in respect to the branching point $r$ and the point reached by characteristic at time $T$. For convenience, as before, the fact that a characteristic reaches set $A$ at time $T$ will be shortly expressed as 'characteristic ends in $A$'.
\begin{figure}[h!] 
\begin{center}
\includegraphics[width=8cm]{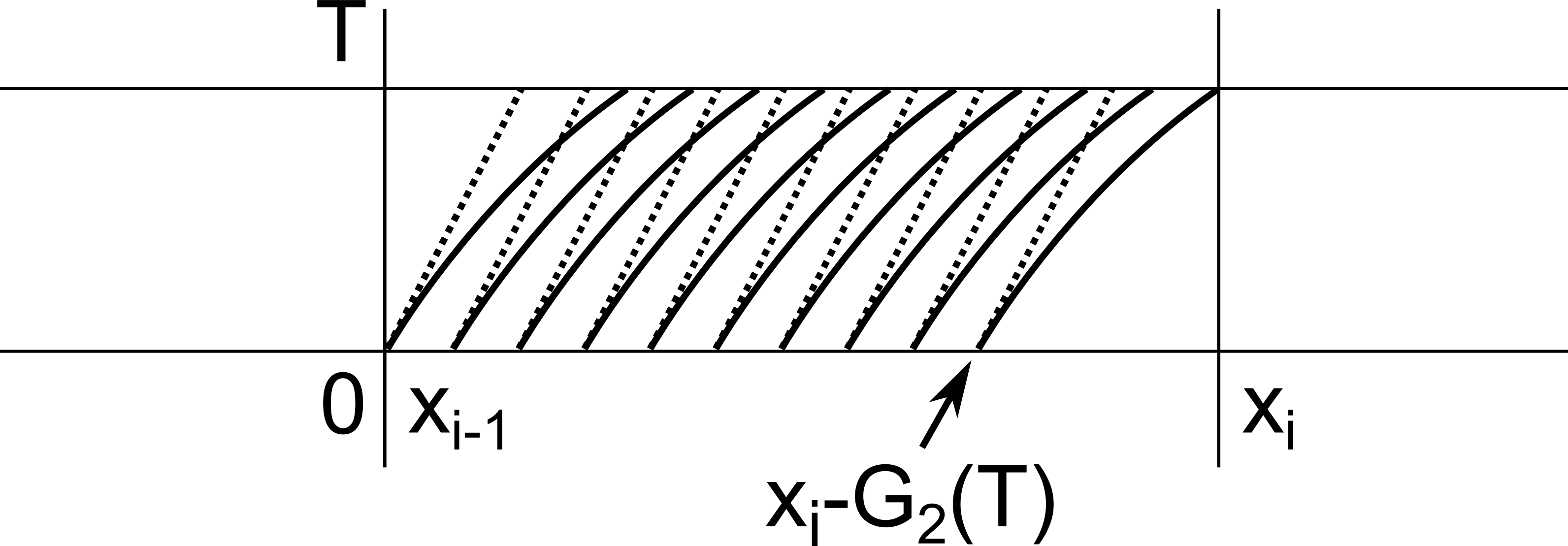}
\caption{Characteristics generated by $g_1(v_1)$ (dotted) and $g_1(v_2)$ (solid) starting in $(x_{i-1}, x_{i} - G_2(T))$ } \label{Fig_char2}
\end{center}
\end{figure}
\begin{figure}[htbp] 
\begin{center}
\includegraphics[width=8cm]{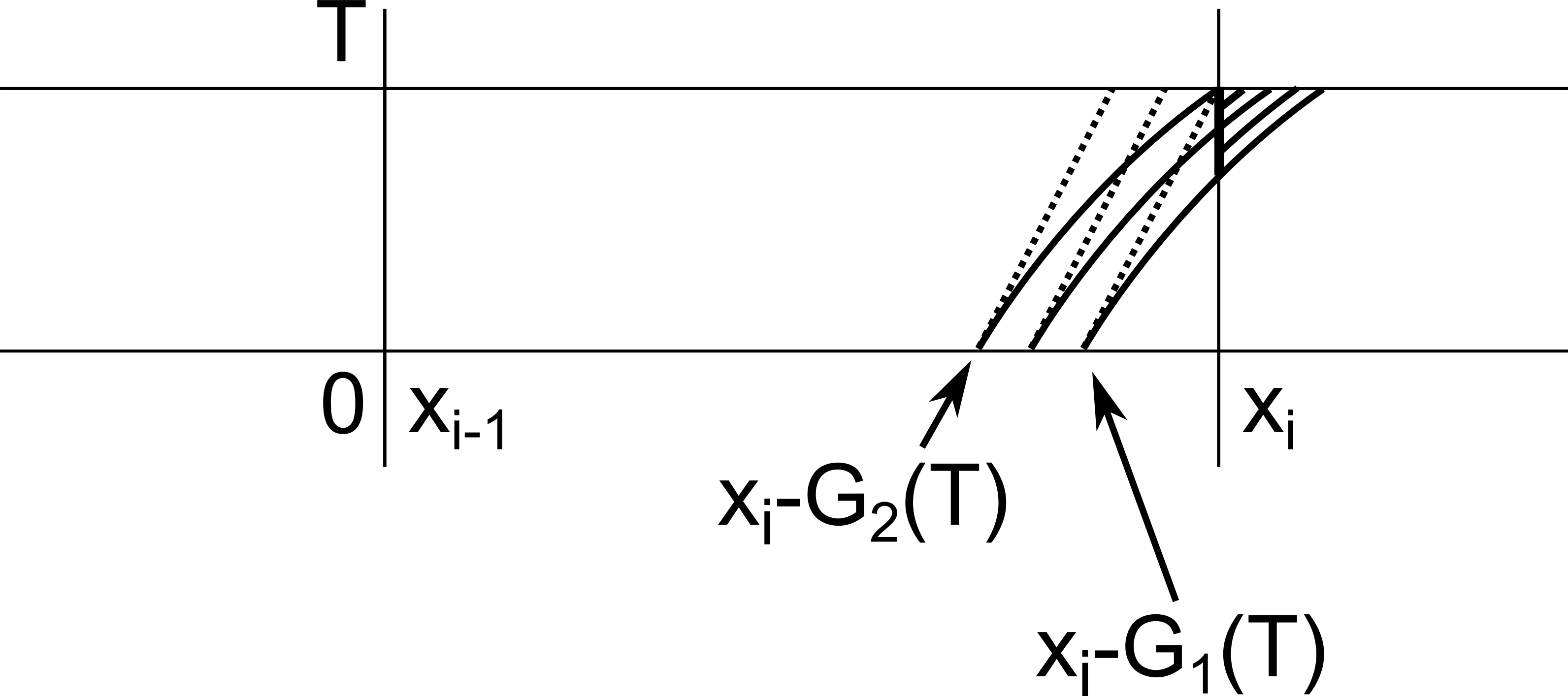} 
\caption{Characteristics generated by $g_1(v_1)$ (dotted) and $g_1(v_2)$ (solid) starting in $[x_{i} - G_2(T), x_{i} - G_1(T))$. Characteristics corresponding to $g_1(v_2)$ arrive in $x_{i}$ before time T and generate fans of characteristics whereas those corresponding to $g_1(v_1)$ do not.} \label{Fig_char3}
\end{center}
\end{figure}
\begin{figure}[h!] 
\begin{center}
\includegraphics[width=4cm]{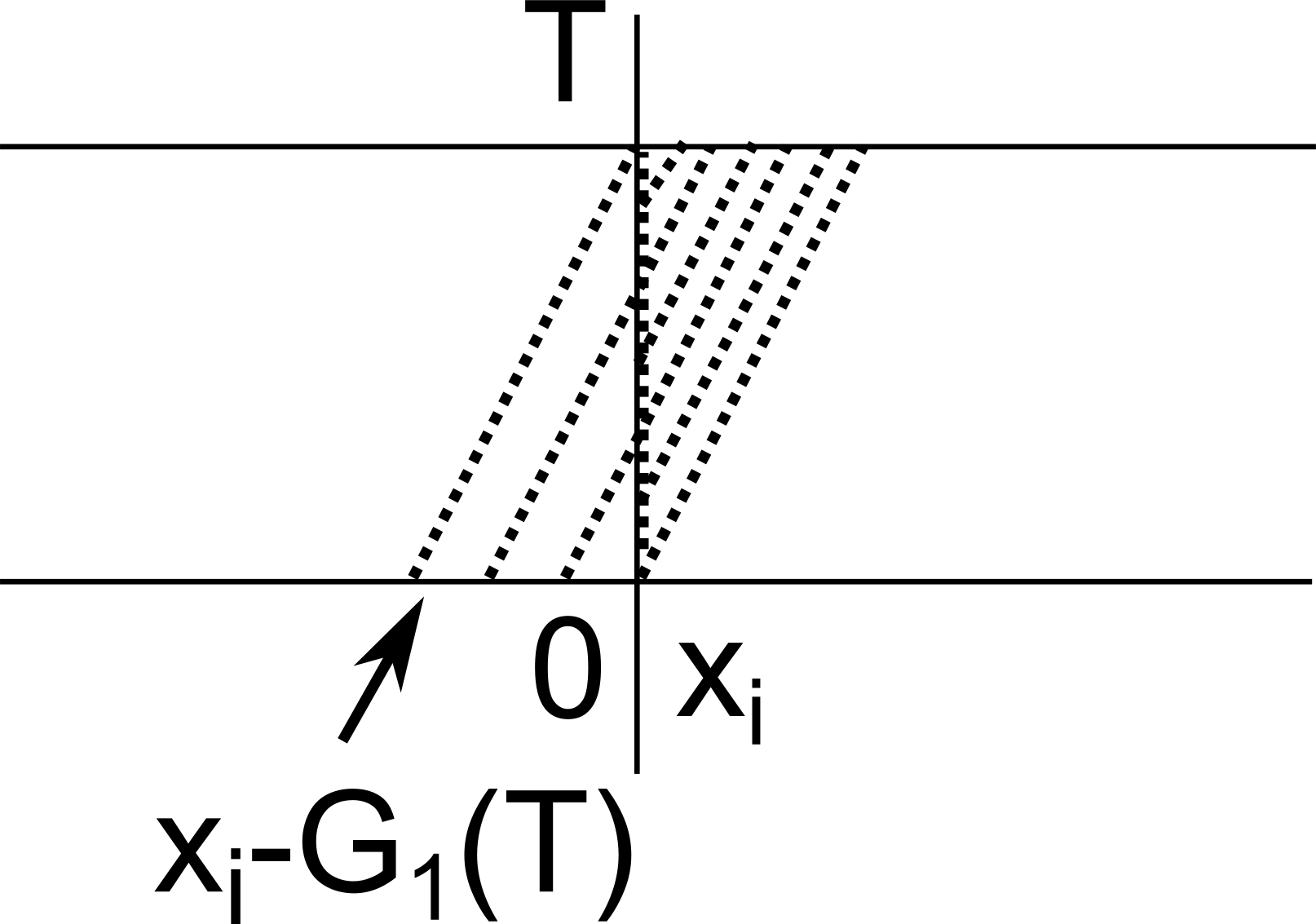}
\includegraphics[width=3.98cm]{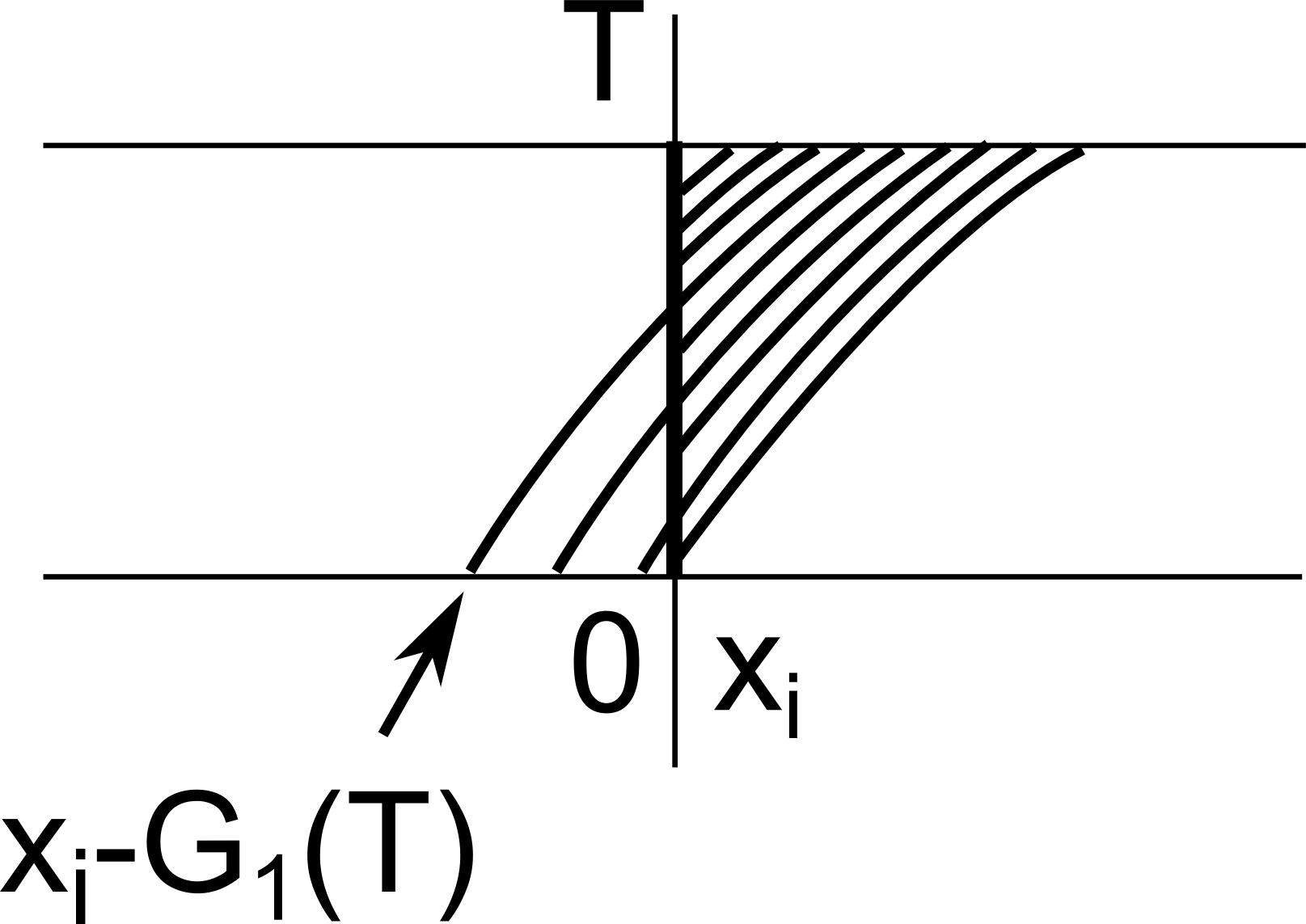}
\includegraphics[width=4.2cm]{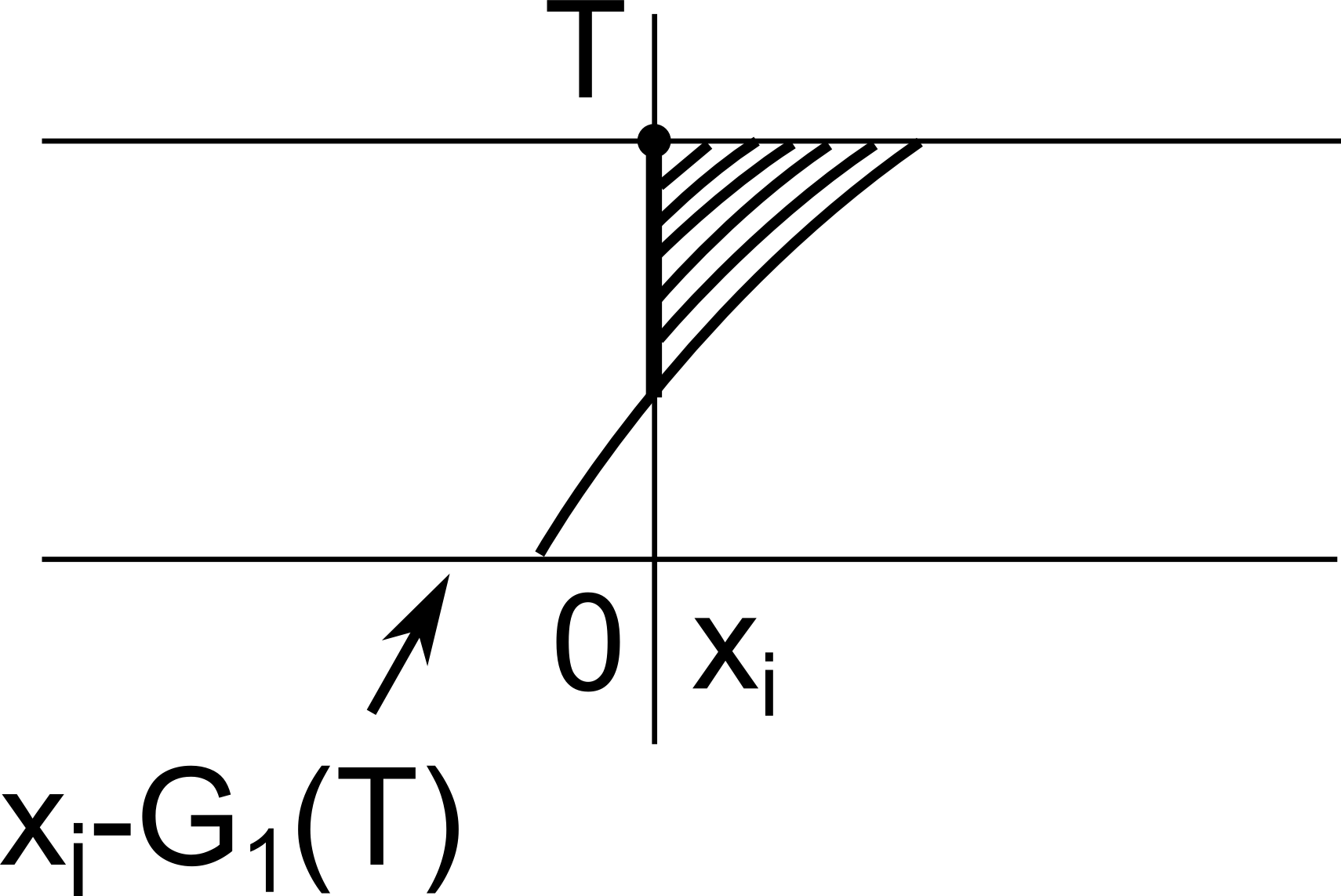}
\caption{Characteristics generated by $g_1(v_1)$ (dotted, left panel) and $g_1(v_2)$ (solid, middle panel) starting in interval $[x_i - G_1(T),x_i]$. After arriving in $x_i$ a given characteristic either spends an arbitrary period of time in $x_i$ before leaving $x_i$ or stays there until time $T$. Thus every characteristic coming to $x_i$ branches generating a fan of characteristics (right panel).}
 \label{Fig_char1}
\end{center}
\end{figure}
\\ \, \\
\noindent {\bf Characteristics starting in $(x_0,x_1-G_2(T))$}
\begin{eqnarray*}
I_1 =  
\int_{(x_0,x_1 - G_2(T))} \left[ \psi \left(x_b+G_2(T)\right) d\mu_2(0)(x_b) - \psi \left(x_b+G_1(T)\right)d\mu_1(0)(x_b) \right] &=& \\ 
\int_{(x_0,x_1 - G_2(T)} \psi \left(x_b + G_1(T)\right) d (\mu_2(0) - \mu_1(0)) &+&\\  \int_{(x_0,x_1 - G_2(T)} \left[\psi\left(x_b + G_2(T)\right) - \psi \left(x_b+G_1(T)\right) \right] d\mu_2(0) &=& U^{I_1} + V^{I_1}. 
\end{eqnarray*}
\noindent {\bf Characteristics starting in $\left[x_1 - G_2(T), x_1-G_1(T) \right)$}\\
For $\mu_1$ these characteristics do not branch before time $T$. 
In case of $\mu_2$, however, they reach $x_1$ before time $T$ and therefore may branch. We obtain
\begin{eqnarray*}
T_1 = \int_{[x_1 - G_2(T), x_1 - G_1(T))} \Bigg\{   \psi(x_1)e^{- \int_{\tau_2(x_b)}^T c_1(v_2(s))ds} d\mu_2(0) + \\ \left(\int_{\tau_2(x_b)}^T e^{- \int_{\tau_2(x_b)} ^r c_1(v_2(s))ds} c_1(v_2(r))\psi\left(x_1 + \int_r^T g_1(v_2(s))ds\right)dr\right) d\mu_2(0) - \psi \left(x_b + G_1(T)\right) d\mu_1(0)  \Bigg\}. 
\end{eqnarray*}
Consecutive terms in the integrand correspond to characteristics related to $\mu_2(0)$ ending in $x_1$, related to $\mu_2(0)$ ending in $(x_1,x_2)$ and related to $\mu_1(0)$. Further calculations lead to
\begin{eqnarray*}
T_1 = \int_{[x_1 - G_2(T), x_1 - G_1(T))} \psi \left(x_b + G_1(T)\right) d(\mu_2(0)-\mu_1(0)) &+& \\
\int_{[x_1 - G_2(T), x_1 - G_1(T))} \left( \psi(x_1) - \psi \left(x_b + G_1(T)\right)\right) d\mu_2(0) &+& \\
\int_{[x_1 - G_2(T), x_1 - G_1(T))} \psi(x_1) \left(e^{-\int_{\tau_2(x_b)}^T c_1(v_2(s))ds} - 1 \right) d\mu_2(0) &+& \\
\int_{[x_1 - G_2(T), x_1 - G_1(T))} \left[ \int_{\tau_2(x_b)}^T c_1(v_2(r)) e^{- \int_{\tau_2(x_b)}^r c_1(v_2(s))ds} \psi\left(x_1 + \int_r^T g_1(v_2(s))ds\right)dr \right] d\mu_2(0) &=& \\
U^{T_1} + V^{T_1}_{1} + V^{T_1}_{2} + V^{T_1}_{3}.&&
\end{eqnarray*}
\noindent{\bf Characteristics starting in $[x_1 - G_1(T),x_1]$}\\
We subdivide those characteristics into three groups, see Fig. \ref{Fig_char11}:
\begin{figure}[htbp]
\begin{center}
\includegraphics[width=3.0cm]{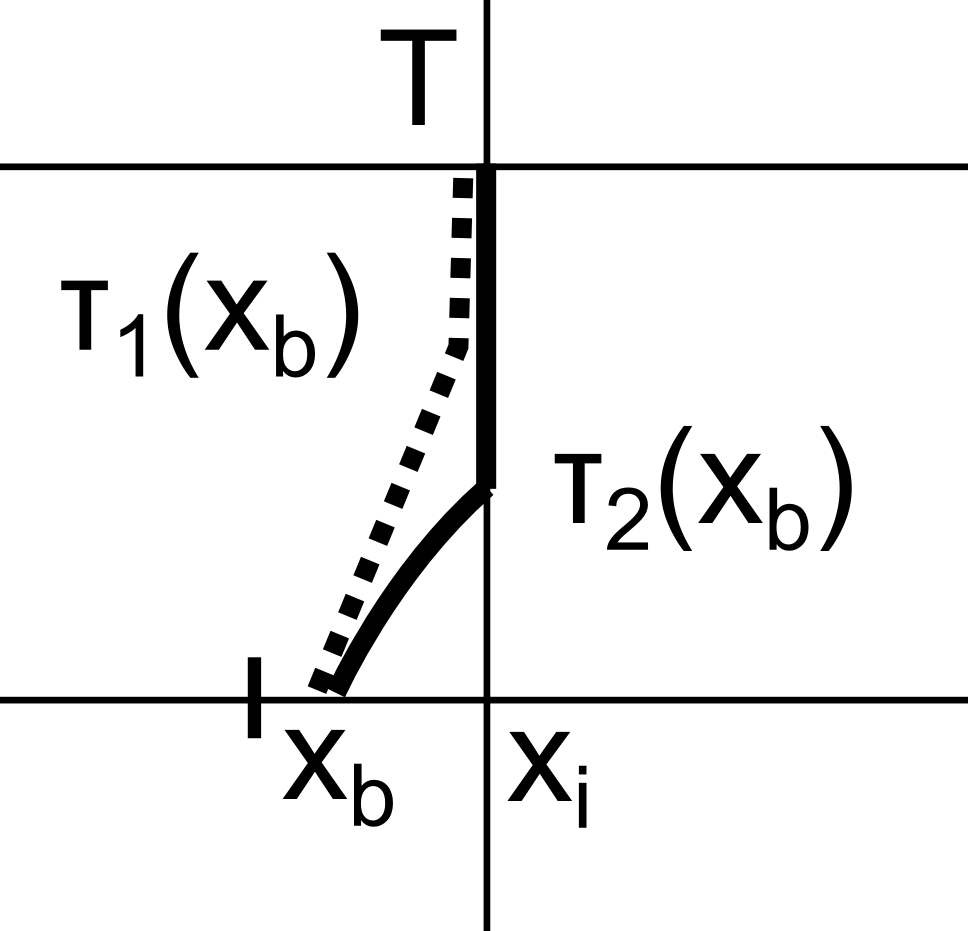} \hspace{1cm}
\includegraphics[width=3.0cm]{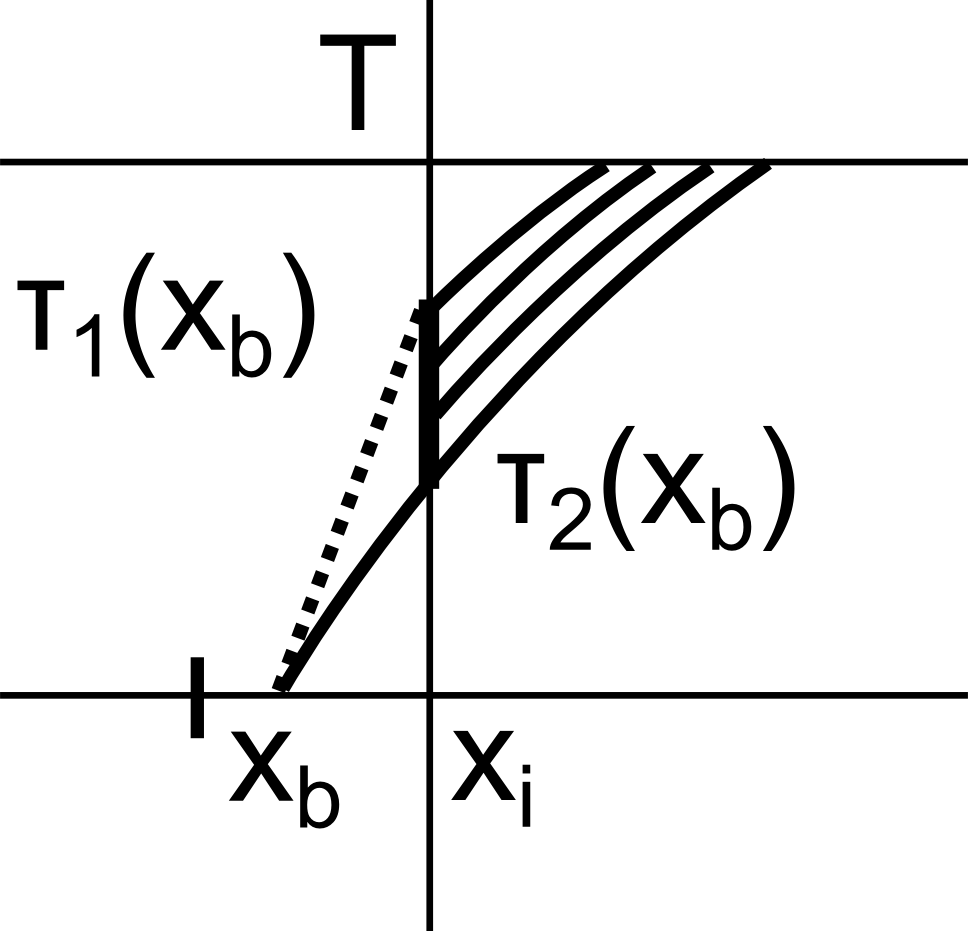} \hspace{1cm}
\includegraphics[width=3.0cm]{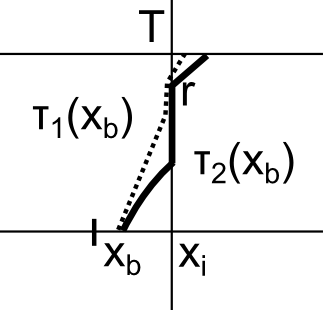}
\end{center}
\caption{Sample characteristics starting in $x_b \in [x_i - G_1(T),x_i]$.
Left panel. Characteristics starting in $x_b$ and both ending in $x_i$. Middle panel. The fan of characteristics arriving at time $\tau_2(x_b)$ and leaving before $\tau_1(x_b)$ is small provided $|\tau_1(x_b) - \tau_2(x_b)|$ is small. Right panel. Characteristics starting in $x_b$ and \emph{both} branching off at time $r$. }
\label{Fig_char11}
\end{figure}

\begin{itemize}
\item those ending in $x_1$, 
\item those ending in $(x_1,x_2)$ and branching off between $\tau_2(x_b)$ and $\tau_1(x_b)$,
\item those ending in $(x_1,x_2)$ and branching off between $\tau_1(x_b)$ and $T$.
\end{itemize}
This leads to:
\begin{eqnarray*}
D_1 = \int_{[x_1 - G_1(T),x_1]} \psi(x_1) e^{- \int_{\tau_2(x_b)}^T c_1(v_2(s))ds} d\mu_2(0)  - \int_{[x_1 - G_1(T),x_1]} \psi(x_1) e^{- \int_{\tau_1(x_b)}^T c_1(v_1(s))ds} d\mu_1(0) + \\
\int_{[x_1 - G_1(T),x_1]} \left( \int_{\tau_2(x_b)}^{\tau_1(x_b)} c_1(v_2(r)) e^{- \int_{\tau_2(x_b)}^r c_1(v_2(s))ds} \psi\left(x_1 + \int_r^T g_1(v_2(s))ds\right)dr \right) d\mu_2(0) + \\
\int _{[x_1 - G_1(T),x_1]} \int_{\tau_1(x_b)}^T \Bigg[ c_1(v_2(r)) e^{- \int_{\tau_2(x_b)}^r c_1(v_2(s))ds} \psi\left(x_1+ \int_r^T g_1(v_2(s))ds\right) dr d\mu_2(0) - \\
c_1(v_1(r)) e^{- \int_{\tau_1(x_b)}^r c_1(v_1(s))ds} \psi\left(x_1 + \int_r^T g_1(v_1(s))ds\right) dr d\mu_1(0)\Bigg] = \\
1^{\circ} + 2^{\circ} + 3^{\circ}.
\end{eqnarray*}
Observe that
\begin{eqnarray*}
1^{\circ} &= &\psi(x_1) \int_{[x_1 - G_1(T),x_1]} e^{- \int_{\tau_1(x_b)}^T c_1(v_1(s))ds} d(\mu_2(0) - \mu_1(0)) + \\
&&\psi(x_1) \int_{[x_1 - G_1(T),x_1]}  \left( e^{- \int_{\tau_2(x_b)}^T c_1(v_2(s))ds}  - e^{-\int_{\tau_1(x_b)}^T c_1(v_1(s))ds} \right)d\mu_2(0) = U^{D_1}_{1} + V^{D_1}_{1}, \\
2^{\circ} &=& V^{D_1}_{2}\\
3^{\circ} &=& \int_{[x_1 - G_1(T),x_1]} \left(\int_{\tau_1(x_b)}^T c_1(v_1(r)) e^{- \int_{\tau_1(x_b)}^r c_1(v_1(s))ds} \psi\left(x_1 + \int_r^T g_1(v_1(s))ds \right)dr \right) d(\mu_2(0) - \mu_1(0)) \\
&+&\int_{[x_1 - G_1(T),x_1]} \int_{\tau_1(x_b)}^T \Bigg[c_1(v_2(r)) e^{-\int_{\tau_2(x_b)}^r c_1(v_2(s))ds} \psi\left(x_1 + \int_r^T g_1(v_2(s))ds\right) -\\
&&    c_1(v_1(r)) e^{- \int_{\tau_1(x_b)}^r c_1(v_1(s))ds} \psi\left(x_1 + \int_r^T g_1(v_1(s))ds\right) \Bigg] dr d\mu_2(0) = U^{D_1}_{2}+V^{D_1}_{3}.
\end{eqnarray*}
\newline
Collecting similar terms we obtain
\begin{eqnarray*}
I_1 + T_1 + D_1 = \left(U^{I_1} + U^{T_1} + U^{D_1} + U^{D_2}\right) + \left(V^{I_1} + V^{T_1}_1 + V^{T_1}_2 + V^{T_1}_3 + V^{D_1}_1 + V^{D_1}_2 + V^{D_1}_3 \right).
\end{eqnarray*}
Next, we estimate U-terms and V-terms using, mostly without explicit reference, Propositions \ref{Est_basic}-\ref{Est_9206}. 

\noindent{\bf U terms}\\
\begin{eqnarray*}
\left(U^{I_1} + U^{T_1} + U^{D_1}_1 + U^{D_1}_2\right) = \int_{(x_0,x_1]} \psi^0(x_b) d(\mu_2(0) - \mu_1(0))(x_b),
\end{eqnarray*}
where
\begin{equation*}
\psi^0(x_b) = 
\begin{cases}
\psi(x_b + G_1(T)) &\mbox{ for } x_0 < x_b < x_1 - G_1(T) \\
\psi(x_1)e^{-\int_{\tau_1(x_b)}^T c_1(v_1(s))ds} + & \\  \int_{\tau_1(x_b)}^T c_1(v_1(r)) e^{- \int_{\tau_1(x_b)}^r c_1(v_1(s))ds} \newline \psi\left(x_1 + \int_r^T g_1(v_1(s))ds \right)dr &\mbox{ for } x_1 - G_1(T) \le x_b \le x_1.
\end{cases}
\end{equation*}
Note that $\psi^0$ is continuous in $x_1 - G_1(T)$ and left-continuous in $x_1$. Let us compute explicitly the derivative of $\psi^0$ for $x_1 - G_1(T) < x_b < x_1$. 
\begin{eqnarray*}
(\psi ^0) ' (x_b) &=& \tau_1'(x_b) c_1(v_1(\tau_1(x_b))) \psi(x_1) e^{- \int_{\tau_1(x_b)}^T c_1(v_1(s))ds}\\
 &-& \tau_1 '(x_b) c_1(v_1(\tau_1(x_b))) \psi \left(x_1 + \int_{\tau_1(x_b)}^T g_1(v_1(s))ds \right)\\
&+&\int_{\tau_1(x_b)}^T c_1(v_1(r)) e^{- \int_{\tau_1(x_b)}^r c_1(v_1(s))ds} \tau_1' (x_b) c_1(v_1(\tau_1(x_b))) \psi \left(x_1 + \int_r^T g_1(v_1(s))ds \right)dr.
\end{eqnarray*}
Using $\tau_1'(x_b) \le \frac 1 {\min(g_1)}$, which is bounded by \eqref{Cor_ming1}, we arrive at
\begin{equation}
|(\psi^0)'(x_b)| \le \frac {1}{\min(g_1)} \left(\sup(\psi)\sup(c_1) + \sup(\psi)\sup(c_1) + \sup(\psi)T(\sup(c_1))^2  \right).
\end{equation}
Similar calculations give analogous estimates for $\psi^0$ on $(x_{i-1},x_i]$ for $i\in \{0,\dots,N\}$. 
Thus,
\begin{eqnarray*}
|\psi^0(x_b)| &\le& \sup(\psi)
\end{eqnarray*}
for all $x_b \in [x_0,x_N]$ and 
\begin{eqnarray*}
\left|\frac {d}{dx_b} \psi^0(x_b)\right| &\le& \max \left(\sup(\psi'), \frac {\sup(\psi)\sup(c)}{\min(g_1)} (2 + T\sup(c))   \right)
\end{eqnarray*}
for $x_b \in (x_0,x_1)\cup (x_1,x_2)\cup \dots \cup(x_{N-1},x_N)$, where $\sup(c) = \max_{i \in\{0,1,\dots,N\}}\sup(c_i)$. 
\\ \, \\
\noindent{\bf V terms}
\begin{eqnarray*}
|V^{I_1}| &\le& {\rm Lip}(\psi) {\rm Lip}(g_1) \mu_2(0)((x_0,x_1 - G_2(T)) \int_0^T |v_2(s) - v_1(s)|ds, 
\end{eqnarray*}
where we used $G_2(T) - G_1(T) \le {\rm Lip}(g_1)\int_0^T |v_2(s) - v_1(s)|ds$. Here and below ${\rm Lip}(\psi)$ is the Lipschitz constant of $\psi$ on interval $(x_{i-1},x_i]$, which is bounded by $1$.
\begin{eqnarray*}
|V^{T_1}_1| &\le& {\rm Lip}(\psi){\rm Lip}(g_1) \mu_2(0)([x_1-G_2(T),x_1 - G_1(T)) \int_0^T |v_2(s) - v_1(s)| ds, \\
|V^{T_1}_2| &\le& |\psi(x_1)| \sup(c_1) \sup ( T - \tau_2(x_b))\mu_2(0)([x_1-G_2(T),x_1 - G_1(T))  \\ &\le& |\psi(x_1)| \sup(c_1) \frac {{\rm Lip}(g_1)}{\min(g_1)} \mu_2(0)([x_1-G_2(T),x_1 - G_1(T)) \int_0^T |v_2(s)-v_1(s)|ds,\\
|V^{T_1}_3| &\le& \sup(\psi) \sup(c_1)  \frac {{\rm Lip}(g_1)}{\min(g_1)} \mu_2(0)([x_1-G_2(T),x_1 - G_1(T)) \int_0^T |v_2(s) - v_1(s)| ds,\\
|V^{D_1}_1| &\le& |\psi(x_1)| \left| \int_{\tau_2(x_b)}^T c_1(v_2(s)ds - \int_{\tau_1(x_b)}^T c_1(v_1(s))ds\right| \mu_2(0)([x_1-G_1(T),x_1]) \\ &\le& 
|\psi(x_1)| \left({\rm Lip}(c_1) \int_0^T |v_2(s) - v_1(s)|ds + \sup(c_1)\sup_{x_b} |\tau_2(x_b) - \tau_1(x_b)| \right)
 \mu_2(0)([x_1-G_1(T),x_1]) \\ &\le&
|\psi(x_1)| \left({\rm Lip}(c_1) + \sup(c_1) \frac {{\rm Lip}(g_1)} {\min(g_1)} \right) \mu_2(0)([x_1-G_1(T),x_1]) \int_0^T |v_2(s) - v_1(s)| ds, \\
|V^{D_1}_2| &\le& \sup_{x_b}|\tau_2(x_b) - \tau_1(x_b)| \sup(c_1) \sup(\psi) \mu_2(0)([x_1-G_1(T),x_1])\\ &\le& \frac{{\rm Lip}(g_1)}{\min(g_1)} \sup(c_1) \sup(\psi) \mu_2(0)([x_1-G_1(T),x_1]) \int_0^T |v_1(s) - v_2(s)|ds.
\end{eqnarray*}
To estimate $V^{D_1}_3$ let us first consider the inner integral
\begin{eqnarray*}
\int_{\tau_1(x_b)}^T \Bigg[c_1(v_2(r)) e^{- \int_{\tau_2(x_b)}^r c_1(v_2(s))ds} \psi\left(x_1 + \int_r^T g_1(v_2(s))ds \right) -\\
 c_1(v_1(r)) e^{- \int_{\tau_1(x_b))}^r c_1(v_1(s))ds} \psi\left(x_1 + \int_r^T g_1(v_1(s))ds \right)   \Bigg] dr = \\
 \int_{\tau_1(x_b)}^T (c_1(v_2(r)) - c_1(v_1(r)) e^{- \int_{\tau_2(x_b)}^r c_1(v_2(s))ds} \psi\left(x_1+ \int_r^T g_1(v_2(s))ds\right) dr + \\
 \int_{\tau_1(x_b)}^T c_1(v_1(r))\left(e^{-\int_{\tau_2(x_b)}^r c_1(v_2(s))ds} - e^{- \int_{\tau_1(x_b)}^r c_1(v_1(s))ds} \right) \psi\left(x_1+ \int_r^T g_1(v_2(s))ds \right)dr + \\
 \int_{\tau_1(x_b)}^T c_1(v_1(r)) e^{-\int_{\tau_1(x_b)}^r c_1(v_1(s))ds} \left(\psi\left(x_1 + \int_r^T g_1(v_2(s))ds \right) - \psi\left(x_1 + \int_r^T g_1(v_1(s))ds\right) \right)dr = \\
 I_\alpha + I_\beta + I_\gamma.
\end{eqnarray*}
Now,
\begin{eqnarray*}
|I_\alpha| &\le& \sup(\psi) {\rm Lip}(c_1) \int_0^T |v_2(s) - v_1(s))| ds \\
|I_\beta| &\le& T\sup (\psi) \sup(c_1) \left(\sup(c_1) \frac {{\rm Lip}(g_1)}{\min(g_1)} +{\rm Lip}(c_1) \right) \int_0^T |v_2(s) - v_1(s)| ds\\
|I_\gamma| &\le& T \sup(c_1){\rm Lip}(\psi){\rm Lip}(g_1) \int_0^T |v_2(s)-v_1(s)|ds,
\end{eqnarray*}
where for $I_{\beta}$ we used the estimate 
\begin{eqnarray*}
\left| e^{- \int_{\tau_2}^r c_1(v_2(s))ds} - e^{- \int_{\tau_1}^r c_1(v_1(s))ds} \right| \le \left| \left(\int_{\tau_2}^r c_1(v_2(s))ds - \int_{\tau_1}^r c_1(v_1(s))ds\right) \right| \le \\
\sup(c_1)\sup|\tau_2 - \tau_1| + {\rm Lip}(c_1)\int_0^T |v_2(s) - v_1(s)|ds \le \\
\left(\frac {\sup(c_1){\rm Lip}(g_1)}{\min(g_1)} + {\rm Lip}(c_1)\right) \int_0^T |v_2(s) - v_1(s)|ds.
\end{eqnarray*}
Thus,
\begin{eqnarray*}
|V^{D_1}_3| &\le& (|I_\alpha| + |I_\beta| + |I_\gamma|) \mu_2(0)([x_1-G_1(T),x_1]) \\ &\le&
 \bigg(\sup(\psi){\rm Lip}(c_1) + \sup(\psi)\sup(c_1)T \left(\sup(c_1) \frac {{\rm Lip}(g_1)}{\min(g_1)} + {\rm Lip}(c_1) \right)\\ && + {\rm Lip}(\psi)\sup(c_1)T {\rm Lip}(g_1) \bigg)  \mu_2(0)([x_1-G_1(T),x_1]) \int_0^T |v_1(s) - v_2(s)|ds.
\end{eqnarray*}
Combining $U$-terms and $V$-terms for $i \in \{0,\dots,N\}$ we obtain
\begin{eqnarray*}
\left| \int_{\mathbb{R}} \psi d(\mu_2(T) - \mu_1(T))\right| \le |D_0| + |I_1 + T_1 + D_1| + |I_2 + T_2 + D_2| + \dots + |I_N + T_N + D_N| \le \\ \int_{\mathbb{R}} \psi^0 d(\mu_2(0)-\mu_1(0)) + \\
\Bigg\{{\rm Lip}(\psi){\rm Lip}(g_1) + 2\sup(\psi)\sup(c)\frac{{\rm Lip}(g_1)}{\min(g_1)} + 2\sup(\psi){\rm Lip}(c) + \\ T\sup(c) \left[\sup(\psi)\left(\sup(c)\frac{{\rm Lip}(g_1)}{\min(g_1)} + {\rm Lip}(c) \right)  + {\rm Lip}(\psi){\rm Lip}(g_1) \right] \Bigg\} TV(\mu_2(0))\int_0^T |v_2(s) - v_1(s)|ds.
\end{eqnarray*}
Above, $TV(\mu_2(0)):= \int_{\mathbb{R}} d\mu_2(0)$.
Taking into account that  $\|\psi^0\|_{W^b_{MT}} \le \|\psi\|_{W^b_{MT}}  \max\left(1, \frac {\sup(c)}{\min(g_1)}(2+T\sup(c))\right) $ we obtain
\begin{eqnarray*}
\rho_{MT} (\mu_1(T),\mu_2(T)) = \sup_{\psi \in B^{1,\infty}_{MT}} \int_{\mathbb{R}} \psi d(\mu_2(T) - \mu_1(T)) \le \\
 \max\left(1, \frac {\sup(c)}{\min(g_1)} (2 + T\sup(c))\right) \rho_{MT}(\mu_1(0),\mu_2(0)) + \\
\left\{{\rm Lip}(g_1) + 2 \sup(c) \frac {{\rm Lip}(g_1)}{\min(g_1)} + 2{\rm Lip}(c) + T\sup(c) \left[\sup(c)\frac {{\rm Lip}(g_1)}{\min(g_1)} + {\rm Lip}(c) + {\rm Lip}(g_1) \right] \right\} \\TV(\mu_2(0)) \int_0^T |v_2(s) - v_1(s)|ds.
\end{eqnarray*}

\noindent This in combination with \eqref{Nonlinear_Estimate} leads to the following local stability result.
\begin{corollary}[\textbf{Local in time stability estimate}]
For $0 < T < T_{max}$, where $T_{max}$ is given by $\eqref{Eq_defTmax2}$, we have
\begin{equation}
\rho_{MT}(\mu_1(T),\mu_2(T)) \le C_1(T) \rho_{MT} (\mu_1(0),\mu_2(0)),
\label{Final_stability_estimate}
\end{equation}
where
\begin{eqnarray}
\label{Final_stability_estimate_constant}
C_1(T) &=& \label{C_1est} \max \left(1,\frac{\sup(c)}{\min(g_1)} (2 + T\sup(c))\right) + \\ 
&&\left\{{\rm Lip}(g_1) + 2 \sup(c) \frac {{\rm Lip}(g_1)}{\min(g_1)} + 2{\rm Lip}(c) + T\sup(c) \left[\sup(c)\frac {{\rm Lip}(g_1)}{\min(g_1)} + {\rm Lip}(c) + {\rm Lip}(g_1) \right] \right\} \times \nonumber \\
&&\times TV(\mu_2(0)) \max \left( \frac {1}{\min(g_1)},T\right) \frac {1}{ \left( 1 - \frac {{\rm Lip}(g_1)}{\min(g_1)} (\mu_1(0)(J_{\max})+ \mu_2(0)(J_{max})) \right)}. \nonumber
\end{eqnarray}
\end{corollary}
The following two examples show that it is impossible to obtain a stability estimate with $C_1(0^+)=1$ for arbitrary initial data.
\begin{example}
\label{Example_ConstantNot1}
Take $\mu(0) = \delta_{x_1}$ and $\mu^{\varepsilon}(0) = \delta_{x_1 - \varepsilon}$ as well as $g_1 \equiv 1$ and $c_1$ constant.\\
Then, $$\rho_{MT}(\mu(0),\mu^{\varepsilon}(0)) = \varepsilon.$$
Let these measures evolve according to equation \eqref{eq_munonlinp0}. For $t = \varepsilon$ we obtain
\begin{eqnarray*}
\mu^{\varepsilon}(t=\varepsilon) &=& \delta_{x_1},\\
\mu(t=\varepsilon) &=& e^{-c_1\varepsilon} \delta_{x_1} + c_1 e^{-c_1(\varepsilon - (x-x_1))}\bold{1}_{[x_1,x_1+\varepsilon]}(x) \mathcal{L}^1(dx).
\end{eqnarray*}
Hence,
\begin{equation*}
\rho_{MT} (\mu(\varepsilon),\mu^{\varepsilon}(\varepsilon)) = 2\left(1- e^{-c_1 \varepsilon}\right) \simeq 2c_1 \varepsilon = 2 c_1 \rho_{MT} (\mu(0),\mu^{\varepsilon}(0)).
\end{equation*}
We conclude that for every $\varepsilon$ there exists a pair of measures $\mu_1(0)$ and $\mu_2(0)$ for which $$\rho_{MT} (\mu_1(\varepsilon),\mu_2(\varepsilon)) \simeq 2 c_1 \rho_{MT} (\mu_1(0),\mu_2(0)).$$
\end{example}
\begin{example}
\label{Ex_1}
Take two initial measures: $$\mu(0) = \delta_{x_N} + \delta_y,$$ $$\mu^{\varepsilon}(0) = \delta_{x_N - \varepsilon} + \delta_y,$$
where $y \in (x_{N-1},{x_N})$ is such that $|x_N - y| > 1$. Clearly, $\rho_{MT}(\mu(0),\mu^{\varepsilon}(0)) = \varepsilon$. 
Take 
\begin{equation*}
g_1(v) = \begin{cases}
\underline{g} & \mbox{for } v=0, \\
1 & \mbox{for } v=1. \\
\end{cases}
\end{equation*}
Let the measures evolve according to equation \eqref{eq_munonlinp0}. 
For $\overline{t}= \frac {\varepsilon}{\underline{g}}$ we obtain $\mu(\overline{t}) = \delta_{x_N} + \delta_{y+\epsilon \slash \underline{g}}$ and $\mu^{\varepsilon}(\tilde{t}) = \delta_{x_N} + \delta_{y+ \varepsilon}$. Thus,
$$\rho_{MT}\left(\mu\left(\frac{\varepsilon}{\underline{g}}\right),\mu^{\varepsilon}\left(\frac{\varepsilon}{\underline{g}}\right)\right) = \varepsilon \left(\frac 1 {\underline{g}} - 1\right) = \left(\frac {1}{\underline{g}} - 1\right) \rho_{MT} (\mu_1(0),\mu_2(0)).$$
Letting $\varepsilon \rightarrow 0$ leads us to conclusion that $C_1(0^+) \ge \left(\frac {1}{\underline{g}} - 1\right)$.
\end{example}

\begin{remark}
\rm
\label{Remark_crosschar}
It may happen that characteristics generated by $g_1(v_1)$ and $g_1(v_2)$ cross in such a way that although $G_1(T) \le G_2(T)$ there exist certain $x_b$ for which $\tau_1(x_b) < \tau_2(x_b)$ (see Fig. \ref{Fig_cross_char}). The reader will easily modify the proof of the linear estimate to encompass such behavior.

\begin{figure}[htbp]
\begin{center}
\includegraphics[width=4.3cm]{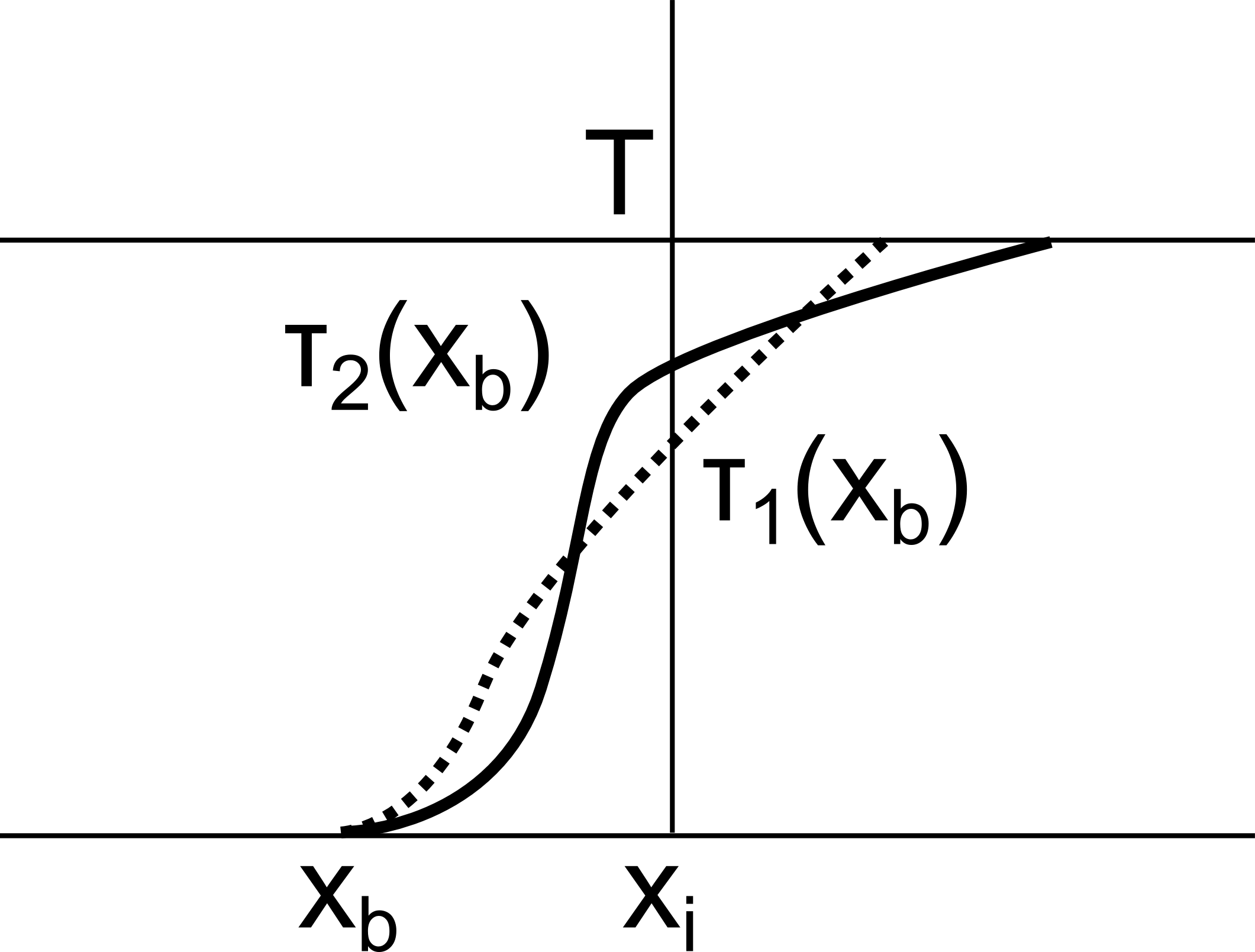}
\caption{Crossing characteristics related to $\mu_1$ (dotted) and $\mu_2$ (solid). Although $\int_0^T g_1(v_1(s))ds < \int_0^T g_1(v_2(s))ds$, there exist certain $x_b$ for which $\tau_1(x_b) < \tau_2(x_b)$.}
\label{Fig_cross_char}
\end{center}
\end{figure}

\end{remark}

\subsection{Stability estimate for large times}
\label{Sec_LargeTimes}
Our goal is to obtain a global in time stability estimate with constant which depends only on the total mass of measures $\mu_1(0), \mu_2(0)$ and not on the initial mass distribution, i.e. the detailed structure of initial measures. We shall iterate estimate \eqref{Final_stability_estimate}. Let, namely, $J_{max}^{t_0}$ be the interval $J_{max}$ corresponding to initial time $t_0$, i.e. 
$J^{t_0}_{max} := (x_N - \max(G_1(t_0,T),G_2(t_0,T)), x_N)$, where $$G_j(t_0,T) := G_j(T) - G_j(t_0)$$ for $j=1,2$ and $G_j$ are given by \eqref{Den_Gj}.
Let, moreover, 
\begin{eqnarray}
&&C_1(t_0,T) := \max \left(1,\frac{\sup(c)}{\min(g_1)} (2 + (T-t_0)\sup(c))\right) +  \label{C_1est2} \\ &&\Bigg\{{\rm Lip}(g_1) + 2 \sup(c) \frac {{\rm Lip}(g_1)}{\min(g_1)} + 2{\rm Lip}(c) + (T-t_0)\sup(c) \left[\sup(c)\frac {{\rm Lip}(g_1)}{\min(g_1)} + {\rm Lip}(c) + {\rm Lip}(g_1) \right] \Bigg\} \times \nonumber \\
&&\times TV(\mu_2(t_0)) \max \left( \frac {1}{\min(g_1)},T-t_0\right) \frac {1}{ \left( 1 - \frac {{\rm Lip}(g_1)}{\min(g_1)} (\mu_1(t_0)(J_{\max}^{t_0})+ \mu_2(t_0)(J_{max}^{t_0})) \right)} \nonumber
\end{eqnarray}
be a generalization of formula \eqref{C_1est} to arbitrary initial times $t_0$.
We choose inductively the time points $0 = T_0 < T_1 < T_2 < \dots$ in such a way that for $j \in \{1,2\}, k = 0,1, \dots$
\begin{equation}
\label{muu_max}
\mu_j(T_k)(J_{max}^{T_k}) \le L := \frac 1 4 \frac {\min(g_1)}{{\rm Lip}(g_1)},
\end{equation}
\begin{equation}
\label{T_max}
\Delta T_k := T_{k+1} - T_k \le \min\left(1, \min_{i \in \{0,\dots,N-1\}} \frac{|x_i-x_{i+1}|}{\sup(g_1)}\right) = \min(1, T_{max}),
\end{equation}
and $\Delta T_k$ are maximal. To obtain the global estimate \eqref{Eq_stabestimate} we observe that 
\begin{itemize}
\item $T_k \to \infty$ as $k \to \infty$,
\item $\rho_{MT}(\mu_1(t),\mu_2(t)) \le  C_1(0,T_1)C_1(T_1,T_2)\times\dots \times C_1(T_{K-1},T_{K}) \rho_{MT}(\mu_1(0),\mu_2(0))$ for $t \in [0,T_K],$ which follows by \eqref{Final_stability_estimate},
\item constants $C_1(T_k,T_{k+1})$ can, by \eqref{C_1est2}-\eqref{T_max}, be bounded in terms of a common constant $\kappa$
, which implies 
\begin{equation}
\label{Eq_pro}
\rho_{MT}(\mu_1(t),\mu_2(t)) \le \kappa^K \rho_{MT}(\mu_1(0),\mu_2(0)).
\end{equation}
\end{itemize}
To finish the proof, we need to estimate, for every given $t>0$, the 'number of iterations' $K$. 
The main difficulty lies in the fact that $\Delta T_k$ are not bounded away from $0$.  The first lemma, which is a consequence of  \eqref{muu_max}-\eqref{T_max}, shows that if $\Delta T_k$ is small, then, informally speaking, the mass which is transported to $x_N$ during the time interval $(T_k,T_{k+1}]$ has to be large.
\begin{lemma}
\label{Lem_a}
Either
$$\Delta T_k = \min\left(1,T_{max}\right)$$ or 
$$\mu_1(T_k)(J^{T_k,lc}_{max}) + \mu_2(T_k)(J^{T_k,lc}_{max}) \ge L,$$ where $lc$ stands for left closure of an interval, i.e. $(a,b)^{lc} = [a,b)$. 
\end{lemma}
\begin{proof}
If $\Delta T_k < \min(1,T_{max})$ then either $$\mu_1(T_k)(J_{max}^{T_k}) = L$$ or $$\mu_2(T_k)(J_{max}^{T_k}) = L$$ or both $$\mu_1(T_k)(J_{max}^{T_k})<L \mbox{ and } \mu_2(T_k)(J_{max}^{T_k}))<L.$$ In the latter case either $\mu_1(T_k)(J_{max}^{T_k,lc}) > L$ or $\mu_2(T_k)(J_{max}^{T_k,lc}) > L$ due to the fact that $\Delta T_k$ is the maximum time interval for which \eqref{muu_max} holds.
\end{proof}

Using Lemma $\ref{Lem_a}$, we estimate the number $It_1$ of iterations which are necessary for the whole mass from interval $(x_{N-1},x_N]$ to 'be transported to $x_N$'.
\begin{lemma}
\label{Lem_numberofsteps}
Let $T_{intmin} := \frac {|x_{N}-x_{N-1}|}{\min(g_1)}$ be the maximum time necessary for all characteristics starting from interval $(x_{N-1},x_N]$ to arrive in $x_N$. Then for 
\begin{equation}
\label{eq_it1}
k \ge It_1:=\left\lceil \frac{T_{intmin}}{\min(1,T_{max})} \right\rceil + \left \lceil \frac {\mu_1(0)((x_{N-1},x_N)) + \mu_2(0)((x_{N-1},x_N))}{L} \right \rceil +1
\end{equation}
 we have $\max(G_1(0,T_k),G_2(0,T_k)) \ge x_N-x_{N-1}$. 
\end{lemma} 
\begin{proof}
Suppose that $\max(G_1(0,T_k),G_2(0,T_k)) < x_N - x_{N-1}$. By formula \eqref{Eq_defmuit} $$\mu_j(T_{k-1})([x_N - G_j(T_{k-1},T_k),x_N)) = \mu_j(0)([x_N - G_j(0,T_k),x_N-G_j(0,T_{k-1}))). $$
Furthermore, by Lemma \ref{Lem_a}, for every $l \in \{0,1,\dots,k\}$ either
\begin{enumerate}[a)]
\item $\Delta T_l = \min(1,T_{max})$, where $T_{max}$ is given by \eqref{Eq_defTmax2} \\
or
\item $\mu_1(0)( [x_N-G_1(0,T_l), x_N - G_1(0,T_{l-1}))) + \mu_2(0)( [x_N-G_2(0,T_l), x_N - G_2(0,T_{l-1}))) \ge L.$
\end{enumerate}
Let 
\begin{itemize}
\item $K_1=\{l \in \{1,\dots,k\} : \mbox{ a) holds}\}$ and
\item $K_2=\{l \in \{1,\dots,k\} : \mbox{ b) holds}\}$.
\end{itemize}
By \eqref{eq_it1} for $k \ge \lceil It_1 \rceil$ either $$\#(K_1) > \left\lceil \frac{T_{intmin}}{\min(1,T_{max})} \right\rceil $$  or $$\#(K_2) > \left \lceil \frac {\mu_1(0)((x_{N-1},x_N)) + \mu_2(0)((x_{N-1},x_N))}{L} \right \rceil,$$ where $\#(K_1), \#(K_2)$ are the numbers of elements of $K_1$ and $K_2$, respectively. This means that either
\begin{eqnarray*}
\mu_1(0)((x_{N-1},x_N)) + \mu_2(0)((x_{N-1},x_N)) \ge \\
\sum_{k \in K_1} \mu_1(0)([x_N-G_1(0,T_k),x_N - G_1(0,T_{k-1}))) + \mu_2(0)([x_N-G_1(0,T_k),x_N - G_1(0,T_{k-1}))) > \\ L \left\lceil \frac {\mu_1(0)((x_{N-1},x_N)) + \mu_2(0)((x_{N-1},x_N))}{L} \right\rceil \ge \mu_1(0)((x_{N-1},x_N)) + \mu_2(0)((x_{N-1},x_N))
\end{eqnarray*}
or
$$T_{intmin} \ge \sum_{k \in K_2} \Delta T_k > \left\lceil \frac {T_{intmin}}{\min(1,T_{max})}\right\rceil \min(1,T_{max}) \ge T_{intmin}.$$ 
In both cases we obtain contradiction.
\end{proof}
Now, using the fact that $TV(\mu_j(t)) = TV(\mu_j(0))$ for all $t>0$  by \eqref{eq_defeta}, we obtain that for $k_0 \ge 0$ and every 
$$k \ge It_2 :=  \left\lceil \frac {T_{intmin}}{\min(1,T_{max})} \right\rceil + \left\lceil \frac {TV(\mu_1(0)) + TV(\mu_2(0))}{L} \right\rceil +2$$
we have
\begin{equation}
\label{Eq_comp1}
\max(G_1(T_{k_0},T_{k_0+k}),G_2(T_{k_0},T_{k_0+k})) > x_N-x_{N-1},
\end{equation}
which follows by Lemma \ref{Lem_numberofsteps} applied for initial time $T_{k_0}$.
On the other hand, for $j=1,2$
\begin{equation}
\label{Eq_comp2}
G_j(T_{k_0}, T_{k_0} + T_{int}) \le x_N - x_{N-1},
\end{equation}
where  $T_{int} :=  {|x_{N-1} - x_N|} \slash {\sup(g_1)}$ and we used the definition of $G_j$.  Comparing \eqref{Eq_comp1} and \eqref{Eq_comp2} we obtain that for every $k_0 \ge 0$  
\begin{equation}
\label{eq_Tk0}
T_{k_0} + T_{int} \le T_{k_0+It_2}.
\end{equation}
Hence, iterating \eqref{eq_Tk0}, we obtain $T_{k_0} + t \le T_{k_0+It_2 \left\lceil {t}\slash{T_{int}} \right\rceil }$ and, in particular, $$t \le  T_{It_2 \left\lceil \frac{t}{T_{int}} \right\rceil } = T_K.$$ 
This, by \eqref{Eq_pro}, leads to the following conclusion.

\begin{corollary}
\begin{equation}
\label{Exponential_Stability_Estimate}
\rho_{MT} (\mu_1(t),\mu_2(t)) \le \kappa^{\left(It_2 \left\lceil \frac{t}{T_{int}}\right\rceil \right)} \rho_{MT}(\mu_1(0),\mu_2(0)),
\end{equation}
where 
\begin{eqnarray*}
T_{int} &=& \frac {|x_{N} - x_{N-1}|}{\sup(g_1)},\\
It_2 &=&  \left\lceil \frac {x_N - x_{N-1}}{\min(g_1)\min\left(1, \min_{i \in \{0,\dots,N-1\}} \frac{|x_i-x_{i+1}|}{\sup(g_1)}\right)} \right\rceil + \left\lceil \frac {TV(\mu_1(0)+TV(\mu_2(0))}{L} \right\rceil +2 ,\\
L &=& \frac 1 4 \frac {\min(g_1)}{{\rm Lip}(g_1)}, \\
\end{eqnarray*}
\begin{eqnarray*}
\kappa &=& \max \left(1,\frac{\sup(c)}{\min(g_1)} (2 + \sup(c))\right) + \\ 
&&2\left\{{\rm Lip}(g_1) + 2 \sup(c) \frac {{\rm Lip}(g_1)}{\min(g_1)} + 2{\rm Lip}(c) + \sup(c) \left[\sup(c)\frac {{\rm Lip}(g_1)}{\min(g_1)} + {\rm Lip}(c) + {\rm Lip}(g_1) \right] \right\} \times \nonumber \\
&&\times TV(\mu_2(0)) \max \left( \frac {1}{\min(g_1)},1\right). 
\end{eqnarray*}
This proves Theorem \ref{StabilityThm}.
\end{corollary}

\begin{corollary}
\label{Cor_Nonlinear_v_Est}
For $T < T_{int}$ we have
\begin{equation*}
\int_0^T |v_1(t) - v_2(t)|dt \le 2 \max\left( \frac{1}{\min(g_1)}, T_{int} \right) (It_1) {\kappa}^{It_2} \rho_{MT} (\mu_1(0),\mu_2(0)).
\end{equation*}
\end{corollary}
\begin{proof}
By \eqref{Nonlinear_Estimate} and \eqref{Exponential_Stability_Estimate} we obtain
\begin{eqnarray*}
\int_{T_k}^{T_{k+1}} |v_1(t) - v_2(t)|dt \le 2\max\left(\frac {1}{\min(g_1)} ,T_{int} \right) \rho_{MT} (\mu_1(T_k),\mu_2(T_{k}))\\ \le 2\max\left(\frac {1}{\min(g_1)} ,T_{int} \right) \kappa^{It_2} \rho_{MT} (\mu_1(0),\mu_2(0)). 
\end{eqnarray*}
Summing from $k=0$ to $k=N-1$, we conclude.
\end{proof}

\begin{remark}
\rm
Time steps in iterations which lead to the global stability estimate \eqref{Exponential_Stability_Estimate} are \emph{different} for every pair of initial measures. This is due to the fact that it is constant 'mass step' that is used rather than constant time step (see Lemma \ref{Lem_a}).
In the end, however, the estimate has the same form for \emph{every} pair of initial measures and depends only on their total variations. 
This is due to the fact that there is a finite potential for small time steps which depends only on the total variation of measures (see Lemma \ref{Lem_numberofsteps}).
\end{remark}

\begin{remark}
\rm
The standard theory of Lipschitz semiflows, see e.g. \cite{BressanHyp}, does not allow us to obtain a global stability estimate from the local one. We recall that a mapping $\Phi: [0,\epsilon] \times [0,T] \times (M,\rho) \rightarrow (M,\rho)$ is called \emph{Lipschitz semiflow} in metric space $(M,\rho)$ if
\begin{itemize}
\item $\Phi(0,t,\mu) = \mu$ for $t\in [0,T]$,
\item  for $t,s_1,s_2$ such that $s_1,s_2,s_1+s_2 \in [0,\epsilon]$ and $t, t+s_1+s_2 \in [0,T]$ 
we have $$\Phi(s_2,t+s_1,\Phi(s_1,t,\mu)) = \Phi(s_1+s_2,t,\mu)$$ (semigroup property),
\item for $t,s_1,s_2$ such that $s_1,s_2 \in [0,\epsilon]$ and $t,t+s_1,t+s_2$ belong to $[0,T]$ we have
\begin{equation}
\label{Est_semiflow}
\rho(\Phi(s_1,t,\mu_1),\Phi(s_2,t,\mu_2)) \le L(\rho(\mu_1,\mu_2) + |s_1-s_2|)
\end{equation}
(Lipschitz continuity).
\end{itemize}
In our case, defining 
\begin{equation}
\label{Eq_Phi}
\Phi(s,t,\mu):= \nu(s), 
\end{equation}
where $\nu(s)$ is the unique solution of problem \eqref{eq_munonlin}-\eqref{eq_icnonlin} with initial condition $\nu(0)=\mu$, we would obtain a semiflow $\Phi$, which would however not be Lipschitz due to the fact that the constant $C_1$ in estimate $\eqref{C_1est}$ cannot be chosen uniformly with respect to $\mu$. Our elementary method of prolongation of the estimate overcomes this difficulty. 
\end{remark}

\begin{remark}
\rm
Stability of measure-transmission solutions of system \eqref{eq_munonlin}-\eqref{eq_icnonlin} with respect to perturbation of the initial condition for general $p$ remains open. 
\end{remark}
\begin{remark}
\label{Rem_whystab}
\rm
The stability result is important from the modelling point of view, since every reasonable model of reality needs to be stable. Moreover, the proof of Theorem \ref{StabilityThm} gives some hints, in the case $p=0$, for construction of a convergent numerical scheme for simulating system \eqref{eq_munonlin}-\eqref{eq_icnonlin}. Such a scheme, based on particle methods, will be presented in a forthcoming paper. 
\end{remark}

\noindent \textbf{Acknowledgements.} The author was partially supported by the International PhD Projects Programme of Foundation for Polish Science operated within the Innovative Economy Operational Programme 2007-2013 (PhD Programme: Mathematical Methods in Natural Sciences). He is grateful to Tomasz Cie\'slak from the Institute of Mathematics, Polish Academy of Sciences for advice and encouragement.

\section*{Appendix. Auxiliary estimates.}

\noindent In this section we gather estimates, which for clarity of the exposition have been omitted from the main text. 
\label{sec_estimates}
\begin{proposition}\ 
\label{Est_basic}
For $x,y \ge 0$ it holds
\begin{enumerate}[i)]
\item $\left|e^x - 1\right| \le |x|e^x$,
\item $\left|e^x - e^y\right| \le |x-y|e^{\max(x,y)}$,
\item $\left| e^{-x} - e^{-y} \right| \le |x-y| e^{-\min(x,y)}$.
\end{enumerate}
\end{proposition}
\begin{proof}
Proof is elementary. 
\end{proof}

\begin{proposition}
\label{Est_3}
Let $\xi(t)$ be an arbitrary non-negative Borel function on $[0,T]$. Then 
\begin{equation*}
\sup_{t\in[0,T]}\left| e^{\xi(t)}-1\right| \le \sup_{t\in[0,T]}\left|  \xi(t)\right |  e^{\sup_{t\in[0,T]}\left| \xi(t)\right |}. 
\end{equation*}
\end{proposition}

\begin{proof}
\begin{equation*}
\sup_{t\in[0,T]}\left| e^{\xi(t)}-1\right| =\sup_{t\in[0,T]}\left| \int_0^{\xi(t)} e^s ds\right|\leq \sup_{t\in[0,T]}\left|  \xi(t)\right |  e^{\sup_{t\in[0,T]}\left| \xi(t)\right | }. 
\end{equation*}
\end{proof}

\begin{proposition}
\label{Est_4}
Let $\tau_1,\tau_2$ be defined by \eqref{def_tauiibis}. For $0<T<T_{max}$ we have
$$|\tau_2(x_b) - \tau_1(x_b)| \le \frac {{\rm Lip}(g_1)}{\min(g_1)} \int_0^T |v_2(s) - v_1(s)|.$$
\end{proposition}
\begin{proof}

\begin{eqnarray*}
|\tau_2(x_b) - \tau_1(x_b)| \le 
\frac {|(x_b+G_2(\min(\tau_1,\tau_2))) - (x_b + G_1(\min(\tau_1,\tau_2)))|}{\min(g_1)} &\le&
\\ \frac {1}{\min(g_1)} \int_0^{\min(\tau_1,\tau_2)} |g_1(v_2(s)) - g_1(v_1(s))|ds &\le&
 \frac {{\rm Lip}(g_1)}{\min(g_1)} \int_0^T |v_2(s) - v_1(s)|ds.
\end{eqnarray*}
\end{proof}

\begin{proposition}
\label{Est_9206}
Let $f^1,f^2$ be two bounded functions defined on interval $[0,T]$. Denote \\ $\sup(f) := \max(\sup(|f^1|), \sup(|f^2|)).$ Then
\begin{equation*}
\left| e^{\int_r^T f^1(s)ds} - e^{\int_r^T f^2(s)ds}\right| \le e^{3T \sup(f)}\int_r^T |f^2(s) - f^1(s)|ds.
\end{equation*}
\end{proposition}

\begin{proof}
\begin{eqnarray*} 
\left| e^{\int_r^T f^1(s)ds} - e^{\int_r^T f^2(s)ds}\right| = \left| e^{\int_r^T f^1(s)ds} \left(1 - e^{\int_r^T (f^2(s) - f^1(s))ds} \right) \right| &\le &\\ 
e^{T \sup(|f^1|)} \left(\int_r^T |f^2(s) - f^1(s)|ds \right) e^{T (\sup(|f^1|)+\sup(|f^2|))} &\le &\\
e^{T (2\sup(|f^1|) + \sup(|f^2|))}\int_r^T |f^2(s) - f^1(s)|ds.
\end{eqnarray*}
\end{proof}

\end{document}